\documentclass[a4paper,12pt, reqno]{amsart}
\usepackage{amsthm, amsmath, amssymb, amsfonts, amsopn, color, mdwlist}
\usepackage{enumerate}
\newcommand{\R}{\mathbb{R}}
\newcommand{\RN}{\mathbb{R}^2}
\newcommand{\intr}{\int_{\mathbb{R}^2}}
\newcommand{\frace}{\frac{1}{\varepsilon^2}}

\newcommand{\e}{\varepsilon}

\newtheorem{theorem}{Theorem}[section]

\newtheorem{remark}[theorem]{Remark}
\newtheorem{prop}[theorem]{Proposition}
\newtheorem{lemma}[theorem]{Lemma}

\numberwithin{equation}{section}

\begin{document}

\title[Asymptotic analysis of  solutions to a  gauged $O(3)$ sigma model]{Asymptotic analysis of  solutions to a  gauged $O(3)$ sigma model}

\author{Daniele Bartolucci}
\address[Daniele Bartolucci]  {University of Rome "Tor Vergata"
Department of Mathematics,
Via della Ricerca Scientifica n.1,
00133 Rome, Italy}
\email{bartoluc@axp.mat.uniroma2.it}

\author{Youngae Lee}
\address[Youngae Lee] {Center for Advanced Study in Theoretical Sciences,
National Taiwan University,
No.1, Sec. 4, Roosevelt Road, Taipei 106, Taiwan}
\email{youngaelee0531@gmail.com}

\author{Chang-Shou Lin}
\address[Chang-Shou Lin] {Taida Institute for Mathematical Sciences,
Center for Advanced Study in Theoretical Sciences,
National Taiwan University,
No.1, Sec. 4, Roosevelt Road, Taipei 106, Taiwan}
\email{cslin@tims.ntu.edu.tw}

\author{Michiaki Onodera}
\address[Michiaki Onodera] {Institute of Mathematics for Industry,
Kyushu University,
744 Motooka, Nishi-ku, Fukuoka 819-0395, Japan}
\email{onodera@imi.kyushu-u.ac.jp}

\begin{abstract}
We analyze an elliptic equation arising in the study of the gauged
$O(3)$ sigma model with the Chern-Simons term.
In this paper, we study the asymptotic behavior of solutions and apply it
to prove the uniqueness of stable solutions. However,
one of the features of this nonlinear equation is the existence of stable nontopological solutions in $\RN$,
which implies
the possibility that a stable solution which blows up at a vortex point exists.
To exclude this kind of blow up behavior is
one of the main difficulties which we have to overcome.
\end{abstract}

\date{\today}
\maketitle
\section{Introduction}
In a recent paper \cite{T1}, G. Tarantello has considered the following  Chern-Simons-Higgs (CSH) model:
\begin{equation}\label{chseq}
\Delta u+\frace e^u(1-e^u)=4\pi\sum^{d}_{j=1}m_{j}\delta_{p_{j}}
\ \mbox{on}\ \Omega,\ \ \
\end{equation}
where $\Omega$ is a flat 2-torus,  $p_{j}$  are distinct points of $\Omega$, and $\delta_p$ stands for the
Dirac measure concentrated at $p$.
Each $p_{j}$ for $1\le j \le d$ is said to be a \textit{vortex} point.
Among other things, the following theorem was proved in \cite{T1}.

\vspace{0.3cm}

{\bf{Theorem A.}} \textit{For given} $\{p_{j}\}$ \textit{and} $m_j\in\mathbb{N}$, \textit{there exists} $\e_0\equiv\e_0(p_{j}, m_{j})>0$ \textit{such that
 if } $\e\in(0,\e_0)$, \textit{then there exists a unique topological solution} $u_\e$ \textit{for (\ref{chseq}), i.e.}
\textit{a unique solution which satisfies} $u_\e\to0$ \textit{a.e. in} $\Omega$ \textit{as} $\e\to0$.

\vspace{0.3cm}

The CSH model has been proposed more than
twenty years ago in \cite{HKP} and independently in \cite{JW} to describe vortices in high temperature
superconductivity. Actually, (\ref{chseq}) was derived from the
Euler-Lagrange equations of the CSH model via a vortex ansatz, see \cite{HKP, JW, T2, Y}.
We also refer to \cite{CK, CKL, LY1, LY2, NT} for more recent developments.

\bigskip

Here we are concerned with another nonlinear equation arising in the study of the gauged $O(3)$ sigma model with the Chern-Simons term:
\begin{equation}\label{maineq}
\Delta u+\frace \frac{e^u(1-e^u)}{(\tau+e^u)^3}=4\pi\sum^{d_{1}}_{j=1}m_{j,1}\delta_{p_{j,1}}-4\pi\sum^{d_{2}}_{j=1}m_{j,2}\delta_{p_{j,2}}
\ \textrm{on}\ \Omega,
\end{equation}
where  $\tau\in(0,\infty)$ is a real parameter.
We assume that $m_{j,i}\ge0$ for all $1\le j \le d_i$ and $i=1, 2$.
Let $$Z\equiv\{p_{j,i}\in\Omega\ |\ 1\le j \le d_i\ \textrm{and}\ i=1,2 \}$$ be the set of vortex points in $\Omega$.
We  define $Z_i\equiv\{p_{j,i}\in\Omega\ |\ 1\le j \le d_i\  \}$ and
$N_i\equiv\sum^{d_{i}}_{j=1}m_{j,i}\ \textrm{for}\ i=1, 2$.
We also adopt the notations $f_\tau(u)\equiv\frac{e^u(1-e^u)}{(\tau+e^u)^3}$ and $f_\tau'(u)=\frac{e^u(\tau-2(\tau+1)e^u+e^{2u})}{(\tau+e^u)^4}$.
For the physical background of this model and other recent studies, we refer to
\cite{ChN, C,  CHLL, CN} and the references quoted therein.

One of the natural questions is whether Theorem A also holds for (\ref{maineq}).
In \cite{T1}, Tarantello proved that if $u_\e$ is a topological solution, then  $u_\e$ is strictly stable solution.
The uniqueness of the topological solutions was established as a consequence of this fact.
 In this paper, we study the uniqueness of \textit{stable solutions}
instead of topological solutions, because the definition of a topological solution depends on a sequence of solutions, not only the solution itself. Here $u$ is called
a stable solution of (\ref{maineq}) if the linearized equation of (\ref{maineq}) at $u$ has nonnegative eigenvalues. In this paper, we prove the equivalence of stable solutions and
 topological solutions under certain assumptions. To state our result, we need the following conditions:
\begin{description}
\item[(H1)] $N_1\neq N_2$;
\item[(H2)] either $\tau=1$ or,
if $N_i>N_k,$ then $m_{j,i}\in[0,1]$  for all $1\le j \le d_i$.
\end{description}
Then we have the following theorem.
 \begin{theorem}\label{stable}  Let $u_\e$ be  a sequence of solutions
of (\ref{maineq}) with $\e>0$. \\
(i) if $u_\e\to 0$ a.e. in $\Omega\setminus Z$ as $\e\to0$, then
$u_\e$ is a strictly stable solution for sufficiently
small  $\e>0$. \\
(ii) if (H1-2) hold and $u_\e$ is a sequence of stable solutions, then $u_\e\to 0$ a.e. in $\Omega\setminus Z$ as $\e\to0$.
\end{theorem}
\begin{remark}
A nontopological  entire solution of the CSH equation (\ref{chseq}) is always unstable (see Appendix).
Hence for a sequence of stable solutions $u_\e$ of the CSH equation (\ref{chseq}), we can prove that $u_\e$ is a topological solution for  small $\e>0$.
The proof  is simpler than  (ii) of Theorem \ref{stable}.
\end{remark}
As a consequence of Theorem \ref{stable}, we also  have the following result about the uniqueness of stable solutions
of (\ref{maineq}).
\begin{theorem}\label{uniqueness}
Let $u_\e$ be a sequence of solutions
of (\ref{maineq}) with $\e>0$.  If (H1-2) hold,  then there exists $\e_0:=\e_0(Z, m_{j,i})>0$ such that
 there exists a unique stable solution of (\ref{maineq}) for each $\e\in(0,\e_0)$.
\end{theorem}

We remark that  the uniqueness of topological solutions of (\ref{maineq}) always holds even without the assumptions  (H1-2).
Indeed, this result and (i) of Theorem \ref{stable} can be proved by a suitable adaptation of the argument in \cite{T1}.
Roughly speaking, this is due to the fact that
the behavior of a topological solution is the same no matter
whether it is a solution of (\ref{chseq}) or of (\ref{maineq}).
See either Proposition 4.8 in \cite{T1} or Lemma \ref{speedofconvergence} below.

However, there are dramatic differences between these two equations when stable solutions are considered.
First of all, the asymptotic analysis is relatively easier for the CSH equation (\ref{chseq}).
By the maximum principle, any solution $u$ of the CSH equation (\ref{chseq}) is
 always negative, thus $e^u(1-e^u)$ is always positive. On the contrary, a solution $u(x)$ of the equation (\ref{maineq})
could tend to either $+\infty$ or $-\infty$ as $x$ converges to a vortex point in case $N_1\neq0$ and $N_2\neq0$.
This fact readily implies that the nonlinear term $f_\tau(u)$ must change sign in $\Omega$ and this is
of course the cause of a lot of difficulties in the study of the asymptotic behavior of $u_\e$ as $\e\to0$.

Secondly, any nontopological entire solution of
the CSH equation (\ref{chseq}) is always unstable.  This might not be  true for
the equation (\ref{maineq}). Indeed, it has been proved that any nontopological   radially symmetric entire solution of (\ref{maineq}) is unstable
provided that either $\tau=1$ or $m_{j,i}\in[0,1]$ for all $i,\ j$. Hence if $\tau\neq1$ and $m_{j,i}>1$ for some $i,\ j$, then there might exist
  nontopological stable entire solutions for (\ref{maineq}). Of course, this fact might complicate our analysis, because stable solutions might be bubbling even at a vortex point $p_{j,i}$, where
  $\tau\neq1$ and $m_{j,i}>1$.  Our condition (H2) partly reflects this fact.
  However, (H2) still
allows the possibility that $m_{j,k}>1$ as far as the global condition $N_i>N_k$ is satisfied, since
in this case one can prove that stable solutions cannot blow up at $p_{j,k}$.
But it is still an interesting open problem to see whether those conditions are necessary or not and
we will discuss it in another paper.

\begin{remark}If any one of the $N_i$'s is zero, then Theorem \ref{stable}-\ref{uniqueness} hold even without the assumptions (H1-2).
\end{remark}

To understand the asymptotic behavior of solutions of (\ref{maineq}) as $\e\to0$,
we also ask whether or not there might exist a sequence of solutions $u_\e$ for (\ref{maineq}) such that
\begin{equation}\label{asympimpossible}
\lim_{\e\to0}\Big(\sup_Ku_\e\Big)=\infty \ \ \textrm{and}\ \  \lim_{\e\to0}\Big(\inf_Ku_\e\Big)=-\infty,
\end{equation}
where $K=\Omega\setminus \cup_{i,j}B_r(p_{j,i})$ for any fixed  $r>0$.
The following theorem tells us that the kind of blow-up behavior as introduced in (\ref{asympimpossible})
cannot occur.
\begin{theorem}\label{BrezisMerletypealternatives}
Let $u_\e$ be a sequence of solutions of (\ref{maineq}). Then, up to  subsequences,
one of the following holds true:

(a) $u_\e\to 0$ uniformly on any compact subset of $\Omega\setminus Z$;

(b) for any compact subset $K\subset\Omega\setminus Z_2$,
there exists $\nu_K>0$ such that $$\lim_{\e\to0}\Big(\sup_K u_\e\Big)\le -\nu_K;$$

(c) for any compact subset $K\subset\Omega\setminus Z_1$,
there exists $\nu_K>0$ such that $$\lim_{\e\to0}\Big(\inf_K u_\e\Big)\ge \nu_K.$$
\end{theorem}
Besides the application to our analysis, we believe that the above  alternative
could be useful in further studies of (\ref{maineq}).

We also remark that it is important to use a suitable Pohozaev type  identity for
handling solutions with different asymptotic behavior.
The following  antiderivatives of $f_\tau(u)$ are used to this purpose depending on the situations at hand:
\begin{equation*}
F_{1,\tau}(u)\equiv\frac{-(1-e^u)^2}{2(\tau+1)(\tau+e^u)^2},
\end{equation*}
and
\begin{equation*}
F_{2,\tau}(u)\equiv\frac{e^u((1-\tau)e^u+2\tau)}{2\tau^2(\tau+e^u)^2}.
\end{equation*}

Moreover, we denote by $G$ the Green's function on $\Omega$ which satisfies
\begin{equation}
-\Delta_x G(x,y)=\delta_y-\frac{1}{|\Omega|},\ x, y\in \Omega\ \textrm{and}\ \int_\Omega G(x,y)dx=0,
\end{equation}
and by $\gamma(x,y)=G(x,y)+\frac{1}{2\pi}\ln|x-y|$ its regular part.
We also  define,
$$
u_0^+(x)\equiv-4\pi\sum^{d_{1}}_{j=1}m_{j,1}G(x,p_{j,1}),\ \ \
u_0^-(x)\equiv-4\pi\sum^{d_{2}}_{j=1}m_{j,2}G(x,p_{j,2}),\ \ \ u_0\equiv u_0^+-u_0^-,
$$
and therefore we see that it holds
\begin{equation}
\Delta u_0=-\frac{4\pi (N_1-N_2)}{|\Omega|}+4\pi\sum^{d_{1}}_{j=1}m_{j,1}\delta_{p_{j,1}}-4\pi\sum^{d_{2}}_{j=1}m_{j,2}\delta_{p_{j,2}}
\ \textrm{on}\ \Omega.
\end{equation}

The rest of this paper is devoted to the proof of the above theorems. In Section 2, we discuss some
preliminary results. In Section 3, we investigate the asymptotic behavior of solutions of (\ref{maineq}) as $\e\to0$.
In Section 4-6, we study the asymptotic behavior of stable solutions.
The main purpose is to prove some identities involving
data coming from different regions, one being a neighborhood of the vortex point and
the other one its complement. The more subtle part is the asymptotic analysis of the bubbling behavior of
stable solutions at vortex points. Finally, we prove Theorems \ref{stable}-\ref{uniqueness}.

\section{Preliminaries}\label{limitingsection}
We consider the following limiting problem for (\ref{maineq}) when $Z$ is empty,
\begin{equation}\label{limitingproinr2}
\Delta u+\frac{e^u(1-e^u)}{(\tau+e^u)^3}=0
\ \textrm{in}\ \RN,
\end{equation}
and we also define (recall $f_\tau(u)=\frac{e^u(1-e^u)}{(\tau+e^u)^3}$)
\begin{equation}\label{defofbeta}
\beta\equiv\frac{1}{2\pi}\intr f_\tau(u)dx.
\end{equation}
By applying the method of moving planes as introduced in \cite{GNN} and improved in  \cite{ChL} and \cite{SY},
we obtain the following lemma.
\begin{lemma}\label{radiallysymm}
Let $u$ be a solution of (\ref{limitingproinr2}). Assume that there exists a constant $c\in\R$ such that
$$\textrm{either} \ \ u\le c\ \  \textrm{or}\ \ u\ge c\ \  \textrm{or}\ \ \limsup_{|x|\to\infty}\frac{|u|}{|x|^2}\le c.$$
If
$f_\tau(u)\in L^1(\RN)$, then $u$ is radially symmetric about some point $x_0\in\RN$.
\end{lemma}
\begin{proof} The proof of Lemma \ref{radiallysymm} is standard and we just provide a sketch for reader's convenience.
First of all, we observe that $f_\tau(u)\in L^1(\RN)\cap L^\infty(\RN)$.
Next we define
\begin{equation}\label{defofv}
v(x)\equiv\frac{1}{2\pi}\intr (\ln|x-y|-\ln(|y|+1))f_\tau(u(y))dy,
\end{equation}
so that $\Delta v=f_\tau(u)$ and by known elliptic estimates
\begin{equation}\label{limofv}
\lim_{|x|\to\infty}\frac{v(x)}{\ln|x|}=\beta.
\end{equation}
At this point we may define $h=u+v$ and then observe that $\Delta h=0$.

\medskip \noindent
\textit{Step 1.} Now we claim that $h$ is constant in $\RN$.
If $u\le c$ or $u\ge c$ in $\RN$ for some constant $c\in\R$ , then (\ref{limofv}) implies that either
$h\le c_1(\ln(|x|+1)+1)$ or $h\ge c_1(\ln(|x|+1)+1)$ for some constant $c_1\in\R$.
Then, by Liouville's theorem, $h(x)=u(x)+v(x)\equiv\textrm{constant}$.
Now we consider the case   $\limsup_{|x|\to\infty}\frac{|u|}{|x|^2}\le c$.
Then, we also see that $\limsup_{|x|\to\infty}\frac{|h|}{|x|^2}$ is bounded.
By the mean value theorem, there exist constants $c_1, c_2\in \R$ such that
$$\sup_{B_{\frac{R}{2}}(y)}|D^\alpha h|\le \frac{c_1}{R^2}\sup_{B_R(y)}|h|\le c_2,$$
for any $y\in\RN$, $R=\frac{|y|}{2}$ and $|\alpha|=2$ (see Theorem 2.10 in \cite{GT}).
Then $D^\alpha h$ is a constant for $|\alpha|=2$ since $D^\alpha h$ is bounded and harmonic in $\RN$.
After a coordinates transformation, we can assume that either
$h(x)=a(x^2_1-x^2_2)+b$ or $h(x)=cx_1+dx_2+e$ for some constants $a, b, c, d, e\in\R$ where $x=(x_1,x_2)$.
Hence (\ref{limofv}) implies that either
\begin{equation}\begin{aligned}\label{uasym1}
u(x)&=a(x^2_1-x^2_2)-(\beta+o(1)) \ln|x|+b
\\&=(a+o(1))(x^2_1-x^2_2)+b
\ \textrm{as}\ |x|\to\infty,
\end{aligned}
\end{equation}
or
\begin{equation}\begin{aligned}\label{uasym2}
u(x)&=cx_1+dx_2-(\beta+o(1)) \ln |x|+e
\\&= (c+o(1))x_1+(d+o(1))x_2+e
\ \textrm{as}\ |x|\to\infty.
\end{aligned}
\end{equation}
For a fixed $\delta\in(0,1)$ we can find a constant $C_{\delta}>0$ such that
\begin{equation}
\begin{aligned}
\infty&>\intr |f_\tau(u)|dx\ge\int_{\delta\le u\le2\delta} |f_\tau(u)|dx
\\&\ge\int_{\delta\le u\le2\delta}C_{\delta} dx
=C_{\delta}|\{x\in\RN\ |\ \delta\le u(x)\le2\delta\}|.
\end{aligned}
\end{equation}
Therefore, by using (\ref{uasym1}) and (\ref{uasym2}),
we see that $|\{x\in\RN\ |\ \delta\le u(x)\le2\delta\}|=\infty$ unless $h$ is constant which proves the claim.
Then, as a consequence of (\ref{limofv}), we see that
\begin{equation}\label{asymuatinf}
\lim_{|x|\to\infty}\frac{u(x)}{\ln|x|}=-\beta.
\end{equation}

\medskip \noindent
\textit{Step 2.}  We claim that if $\beta=0$ then $u\equiv 0$.
Suppose that  there exists $x_0\in\RN$ such that $u(x_0)<0$.
Then there exists $r>0$ such that
\begin{equation}\label{lessat}
u|_{B_{r}(x_0)}< 0.
\end{equation}
Let us set $v_\delta (x)=\delta \ln(\frac{|x-x_0|}{r})$ on $\RN\setminus B_{r}(x_0)$.
Then we see that $v_\delta\ge u$ on $\partial B_r(x_0)$.
Since $u=o(\ln |x|)$ as $|x|\to\infty$ (which is of course a consequence of (\ref{asymuatinf}) and $\beta=0$),
then there exists $R_\delta>0$
such that $v_\delta>u$ on $\RN\setminus B_{R_\delta}(0)$.
We claim that $v_\delta\ge u $ on $B_{R_\delta}(0)\setminus B_r(x_0)$.
If not, there exists $x_1\in B_{R_\delta}(0)\setminus B_r(x_0)$ such that $u(x_1)-v_\delta(x_1)=\max_{B_{R_\delta}(0)\setminus B_r(x_0)}(u-v_\delta)>0$.
Then by the maximum principle, we see that
$$0\ge\Delta(u-v_\delta)(x_1)=-f_\tau(u(x_1))>0\ \textrm{since}\ u(x_1)>v_\delta(x_1)\ge 0.$$
Thus, $v_\delta=\delta\ln(\frac{|x-x_0|}{r})\ge u$ on $\RN\setminus B_r(x_0)$.
Since $\delta >0$ is arbitrary, we conclude that
\begin{equation}\label{outside}
u(x)\le 0\ \textrm{on}\  \RN\setminus B_r(x_0).
\end{equation}
Now we see that (\ref{lessat}) and  (\ref{outside}) together contradict (\ref{defofbeta}) with $\beta=0$.
Therefore we have $u\ge 0$ on $\RN$, and then, by using (\ref{defofbeta}) together with $\beta=0$,
we conclude that $u\equiv 0$  on $\RN$.

\medskip \noindent
\textit{Step 3.} From now on, we consider the case $\beta\neq0$.
By using the strong maximum principle and (\ref{asymuatinf}), we conclude that
 \begin{equation}\begin{aligned}\label{forkelvintrasformation}
\left \{
\begin{array}{ll}
u>0, \ f_\tau(u)<0\ \ \textrm{if}\ \beta<0,
\\u<0, \ f_\tau(u)>0\ \ \textrm{if}\ \beta>0.
\end{array}\right.
\end{aligned}\end{equation}
In view of (\ref{forkelvintrasformation}), we can use the maximum principle to show that
 \begin{equation}\begin{aligned}\label{fromit}
\left \{
\begin{array}{ll}
u\le-\beta\ln|x|+C\ \ \textrm{if}\ \beta<0,
\\u\ge-\beta\ln|x|+C\ \ \textrm{if}\ \beta>0,
\end{array}\right.
\end{aligned}\end{equation}
for large $|x|$ and a suitable constant $C\in\R$.
By using (\ref{fromit}), then $f_\tau(u)\in L^1(\RN)$ implies that $|\beta|>2$ and then we deduce
the sharper estimate
 \begin{equation}
\label{startclaim}
u(x)=-\beta\ln|x|+O(1)\ \ \textrm{as}\ \ |x|\to+\infty.
\end{equation}
At this point, the method of moving planes to be used together with (\ref{startclaim})
shows that $u$ is radially symmetric. Since the proof is standard we skip it here and refer to
\cite{ChL, GNN} for further details.
Therefore, the proof of Lemma \ref{radiallysymm} is completed.
\end{proof}
Let $u(r; s)$ be the solution of the following initial value problem
 \begin{equation}\begin{aligned} \label{limitingpro}
\left \{
\begin{array}{ll}
u''+\frac{1}{r}u'+\frac{e^u(1-e^u)}{(\tau+e^u)^3}=0\ \ \ \textrm{for}\ r>0,
\\ u(0; s)=s, \  u'(0; s)=0,
\end{array}\right.
\end{aligned}\end{equation}
where $u'$ denotes $\frac{du}{dr}(r; s)$ and let us set
\begin{equation}\label{defbetas}
\beta(s)\equiv\frac{1}{2\pi}\intr f_\tau(u(r; s))=\int^\infty_0 f_\tau(u(r; s))rdr.
\end{equation}
It turns out that the solutions of (\ref{limitingpro}) admit only three kinds of limiting conditions as $r\to \infty$:
 \begin{equation}\begin{aligned} \label{boundarycondition}
\left \{
\begin{array}{ll}
\textrm{topological boundary condition}: u\to 0,
\\
\\ \textrm{nontopological boundary condition of type I}:  u\to -\infty,
\\
\\ \textrm{nontopological boundary condition of type II}: u\to\infty.
\end{array}\right.
\end{aligned}\end{equation}

\

We will use the following lemma recently obtained in \cite{CHLL}.
\begin{lemma}\label{propertyofentiresolution}
Let $u(r; s)$ be a solution of (\ref{limitingpro}). Then, we have

(i) $\beta(0)=0$. In this case, $u(r; 0)\equiv 0$ is the unique topological solution of (\ref{limitingpro});

(ii) $\beta: (-\infty, 0)\rightarrow(4,\infty)$ is strictly increasing and bijective and
$$\lim_{s\to 0_-}\beta(s)=\infty\ \textrm{and}\  \lim_{s\to-\infty}\beta(s)=4.$$
In this case, $u(r; s)$ is a nontopological solution of type I;

(iii) $\beta: (0,\infty)\rightarrow(-\infty,-4)$ is strictly increasing and bijective and
$$\lim_{s\to 0_+}\beta(s)=-\infty\ \textrm{and}\  \lim_{s\to\infty}\beta(s)=-4.$$
In this case, $u(r; s)$ is a nontopological solution of type II.
\end{lemma}

\section{Proof of Theorem \ref{BrezisMerletypealternatives}: the asymptotic behavior of solutions}
One of the main steps in the proof of
Theorem \ref{BrezisMerletypealternatives} is to obtain a uniform bound for
\begin{equation*}
\int_{\Omega}\Big|\frace\frac{e^{u_\e}(1-e^{u_\e})}{(\tau+e^{u_\e})^3}\Big|dx.
\end{equation*}
Toward this goal we have the following lemma.
\begin{lemma}\label{bddofintegration} Let $u_\e$ be a sequence of solutions of (\ref{maineq}).
Then, there exists a constant $M_0\in(0,\infty)$ such that
$$\int_{\Omega}\Big|\frace\frac{e^{u_\e}(1-e^{u_\e})}{(\tau+e^{u_\e})^3}\Big|dx\le M_0.$$
\end{lemma}
\begin{proof} We observe that, for any $a\in(0,\infty)$, it holds
\begin{equation}
\begin{aligned}\label{squareint0}
\Delta u_\e\Big(\frac{1-e^{u_\e}}{a+e^{u_\e}}\Big)=\textrm{div}\Big[\nabla u_\e\Big(\frac{1-e^{u_\e}}{a+e^{u_\e}}\Big)\Big]+\frac{(a+1)|\nabla u_\e|^2 e^{u_\e}}{(a+e^{u_\e})^2}.
\end{aligned}
\end{equation}
Then, multiplying both sides of the equation (\ref{maineq}) by $\frac{1-e^{u_\e}}{a+e^{u_\e}}$
and integrating over $\Omega$, we conclude that
\begin{equation}
\begin{aligned}\label{squareint}
 \int_\Omega\frac{(a+1)|\nabla u_\e|^2 e^{u_\e}}{(a+e^{u_\e})^2}+\frac{1}{\e^2}\frac{e^{u_\e}(1-e^{u_\e})^2}{(\tau+e^{u_\e})^3(a+e^{u_\e})}dx=4\pi(\frac{N_1}{a}+N_2).
\end{aligned}
\end{equation}
Let us fix $a=1$.
Then there exist some constants $M_1, M_2\ge0$ such that
 \begin{equation}
 \label{intgramulti}
\int_{\Omega}\frac{|\nabla u_\e|^2e^{u_\e}}{(1+e^{u_\e})^2}dx\le M_1,
 \end{equation}
and
 \begin{equation}
 \label{squareintbdd}
 \int_{\Omega}\frace\frac{e^{u_\e}(1-e^{u_\e})^2}{(\tau+e^{u_\e})^3(1+e^{u_\e})}dx\le M_2.
 \end{equation}

We also see that there exists $\delta_{\e,1}\in(1,2)$ such that
 \begin{equation}
\int_{\{u_\e=-\delta_{\e,1}\}}|\nabla u_\e|dS=\int^{-1}_{-2}\Big(\int_{\{u_\e=r\}}|\nabla u_\e|dS\Big)dr,
\end{equation}
and  there exists a constant $c_0>0$ such that
 \begin{equation}
\int_{\{-2\le u_\e\le0\}}|\nabla u_\e|^2dx\le c_0\int_{\{-2\le u_\e\le0\}}\frac{|\nabla u_\e|^2e^{u_\e}}{(1+e^{u_\e})^2}dx\le c_0M_1.
\end{equation}
Hence we also have
\begin{equation}
\begin{aligned}\label{uppnor}
\int_{\{u_\e=-\delta_{\e,1}\}}|\nabla u_\e|dS
&=\int^{-1}_{-2}\Big(\int_{\{u_\e=r\}}|\nabla u_\e|dS\Big)dr
\\&=\int_{\{-2\le u_\e\le-1\}}|\nabla u_\e|^2dx
\le c_0M_1.
\end{aligned}
\end{equation}
Let $\nu$ be  an exterior unit normal vector to $\partial\{x\in\Omega\ |\ -\delta_{\e,1}\le u_\e\le0\}$.
By using $\frac{\partial u_\e}{\partial \nu}\Big|_{u_\e=0}\ge0$ and (\ref{uppnor}),
we see that
\begin{equation}
\begin{aligned}\label{deltaepminus1}
0&\le\int_{\{-\delta_{\e,1}\le u_\e\le0\}}\frace \frac{e^{u_\e}(1-e^{u_\e})}{(\tau+e^{u_\e})^3}dx
=-\int_{\{-\delta_{\e,1}\le u_\e\le0\}}\Delta u_\e dx
\\&=-\int_{\{u_\e=-\delta_{\e,1}\}}\frac{\partial u_\e}{\partial \nu}dS-\int_{\{u_\e=0\}}\frac{\partial u_\e}{\partial \nu}dS
\\&\le \int_{\{u_\e=-\delta_{\e,1}\}}|\nabla u_\e|dS
\le c_0M_1.
\end{aligned}
\end{equation}
The same argument with minor changes shows that we can find constants $\delta_{\e,2}\in(1,2)$ and $c_1>0$ such that
\begin{equation}\label{deltaepplus1}
\begin{aligned}
\int_{\{0\le u_\e\le \delta_{\e,2}\}}\Big|\frace \frac{e^{u_\e}(1-e^{u_\e})}{(\tau+e^{u_\e})^3}\Big|dx
\le c_1M_1.
\end{aligned}
\end{equation}
Moreover, there exist constants $c_2, c_3>0$ such that
\begin{equation}\label{deltaepminus2}
\begin{aligned}
&\int_{\{u_\e\le-\delta_{\e,1}\}}\Big|\frace \frac{e^{u_\e}(1-e^{u_\e})}{(\tau+e^{u_\e})^3}\Big|dx
\le c_2\int_{\{u_\e\le-\delta_{\e,1}\}}\frace \frac{e^{u_\e}(1-e^{u_\e})^2}{(\tau+e^{u_\e})^3(1+e^{u_\e})}dx
\le c_2M_2,
\end{aligned}
\end{equation}
and
\begin{equation}\label{deltaepplus2}
\begin{aligned}
&\int_{\{u_\e\ge\delta_{\e,2}\}}\Big|\frace \frac{e^{u_\e}(1-e^{u_\e})}{(\tau+e^{u_\e})^3}\Big|dx
\le c_3\int_{\{u_\e\ge\delta_{\e,2}\}}\frace \frac{e^{u_\e}(1-e^{u_\e})^2}{(\tau+e^{u_\e})^3(1+e^{u_\e})}dx
\le c_3M_2.
\end{aligned}
\end{equation}
The desired conclusion follows by using
(\ref{deltaepminus1}), (\ref{deltaepplus1}), (\ref{deltaepminus2}) and (\ref{deltaepplus2}).
\end{proof}
Let us recall the following form of the Harnack inequality which will be widely used in the sequel
(see \cite{BT} and \cite{GT}).
\begin{lemma}\label{harnarkineq}
Let $D\subseteq\RN$ be a smooth bounded domain and $v$ satisfy:
$$-\Delta v=f\ \textrm{in}\ D,$$
with $f\in L^p(D)$, $p>1$. For any subdomain $D'\subset\subset D$, there exist two positive constants $\sigma\in(0,1)$ and $\gamma>0$,
depending on $D'$ only such that:

$$(a)\ \textrm{if}\ \sup_{\partial D} v\le C,\ \textrm{then}\ \sup_{D'} v\le\sigma\inf_{D'}v+(1+\sigma)\gamma\|f\|_{L^p}+(1-\sigma)C,$$

$$(b)\ \textrm{if}\ \inf_{\partial D} v\ge -C,\ \textrm{then}\ \sigma\sup_{D'} v\le\inf_{D'}v+(1+\sigma)\gamma\|f\|_{L^p}+(1-\sigma)C.$$
\end{lemma}
Moreover, we have the following lemmas.
\begin{lemma}\label{green}Let $u_\e$ be a sequence of solutions of (\ref{maineq}).
Let $K$ be a compact subset such that $K\subset\Omega\setminus Z$.
Then there exist constants $a,b>0$ such that $|u_\e(x_\e)-u_\e(z_\e)|\le ar^2+b$ for any $r>0$ and  $z_\e\in B_{\e r}(x_\e)\subseteq K$.
\end{lemma}
\begin{proof}
By using the Green's representation formula for a solution $u_\e$ of (\ref{maineq}),
we see that for $x\in K\subset\subset\Omega\setminus Z$,
\begin{equation}
\begin{aligned}\label{greenrp}
u_\e(x)
&=\frac{1}{|\Omega|}\int_\Omega u_\e(y)dy+\int_\Omega G(x,y)(-\Delta u_\e(y))dy
\\&=\frac{1}{|\Omega|}\int_\Omega u_\e(y)dy
\\&+\int_\Omega G(x,y)\Big(\frace \frac{e^{u_\e}(1-e^{u_\e})}{(\tau+e^{u_\e})^3}-4\pi\sum^{d_{1}}_{j=1}m_{j,1}\delta_{p_{j,1}}+4\pi\sum^{d_{2}}_{j=1}m_{j,2}\delta_{p_{j,2}}\Big)dy
\\&=\frac{1}{|\Omega|}\int_\Omega u_\e(y)dy
+\int_\Omega G(x,y)\frace \frac{e^{u_\e}(1-e^{u_\e})}{(\tau+e^{u_\e})^3}dy+O(1).
\end{aligned}
\end{equation}
Then, \begin{equation}
\begin{aligned}
u_\e(x)-u_\e(z)=\int_\Omega (G(x,y)-G(z,y))\frace\frac{e^{u_\e}(1-e^{u_\e})}{(\tau+e^{u_\e})^3}dy+O(1)
\\ \textrm{for}\ x,z\in K.
\end{aligned}
\end{equation}
In view of  Lemma \ref{bddofintegration}, we see that
 \begin{equation}
\begin{aligned}
u_\e(x)-u_\e(z)=\frac{1}{2\pi\e^2}\int_\Omega \ln\Big(\frac{|z-y|}{|x-y|}\Big)\frac{e^{u_\e}(1-e^{u_\e})}{(\tau+e^{u_\e})^3}dy+O(1)
\\ \textrm{for}\ x,z\in K.
\end{aligned}
\end{equation}
For fixed $r>0$, we assume that $z_\e\in B_{\e r}(x_\e) \subseteq  K$.
By the mean value theorem, there exists $\theta=\theta(\e,y)\in(0,1)$ such that
 \begin{equation}
\begin{aligned}
|\ln|z_\e-y|-\ln|x_\e-y||&=\frac{||z_\e-y|-|x_\e-y||}{\theta|z_\e-y|+(1-\theta)|x_\e-y|}
\\&\le\frac{|x_\e-z_\e|}{\theta|z_\e-y|+(1-\theta)|x_\e-y|}.
\end{aligned}
\end{equation}
For any $y\in\Omega\setminus B_{2\e r}(x_\e)$, we have $|z_\e-y|\ge \e r$ and $|x_\e-y|\ge 2\e r$.
Thus, we see that
\begin{equation}
\begin{aligned}
|\ln|z_\e-y|-\ln|x_\e-y||\le\frac{\e r}{\theta \e r+(1-\theta)2\e r}=\frac{1}{2-\theta}\le1\ \textrm{on}\ \Omega\setminus B_{2\e r}(x_\e).
\end{aligned}
\end{equation}
At this point, Lemma \ref{bddofintegration} implies that
\begin{equation}
\begin{aligned}
 \frac{1}{2\pi\e^2}\int_{\Omega\setminus B_{2\e r}(x_\e)}\Big|\ln\Big(\frac{|z_\e-y|}{|x_\e-y|}\Big)\frac{e^{u_\e}(1-e^{u_\e})}{(\tau+e^{u_\e})^3}\Big|dy=O(1).
\end{aligned}
\end{equation}
We also see that
\begin{equation}
\begin{aligned}
\int_{B_{2\e r}(x_\e)}\Big|\ln\Big(\frac{|z_\e-y|}{|x_\e-y|}\Big)\Big|dy
&\le\int_{B_{2\e r}(x_\e)}\frac{|x_\e-z_\e|}{\theta|z_\e-y|+(1-\theta)|x_\e-y|}dy
\\&\le\int_{B_{2\e r}(x_\e)}\frac{|x_\e-z_\e|}{\min\{|z_\e-y|,|x_\e-y|\}}dy
\\&\le\int_{B_{2\e r}(x_\e)}\frac{|x_\e-z_\e|}{|z_\e-y|}+\frac{|x_\e-z_\e|}{|x_\e-y|}dy
\\&\le\int_{B_{4\e r}(z_\e)}\frac{|x_\e-z_\e|}{|z_\e-y|}dy
+\int_{B_{2\e r}(x_\e)}\frac{|x_\e-z_\e|}{|x_\e-y|}dy
\\&\le2\int_{B_{4\e r}(0)}\frac{|x_\e-z_\e|}{|y|}dy
\le16r^2\e^2\pi.
\end{aligned}
\end{equation}
Therefore we conclude that
\begin{equation}
\begin{aligned}
&\frac{1}{2\pi\e^2}\int_{B_{2\e r}(x_\e)}\Big|\ln\Big(\frac{|z_\e-y|}{|x_\e-y|}\Big)\frac{e^{u_\e}(1-e^{u_\e})}{(\tau+e^{u_\e})^3}\Big|dy
\le8r^2\sup_{t\in\mathbb{R}}\Big|\frac{e^t(1-e^t)}{(\tau+e^t)^3}\Big|,
\end{aligned}
\end{equation}
and we readily obtain constants $a,b>0$ such that for any $r>0$, it holds
\begin{equation}
\begin{aligned}
|u_\e(x_\e)-u_\e(z_\e)|\le ar^2+b\  \textrm{for}\ z_\e\in B_{\e r}(x_\e)\subseteq K.
\end{aligned}
\end{equation}
\end{proof}
\begin{lemma}\label{uniformestimates}
Let $K$ be a connected  compact set such that $K\subset \Omega\setminus Z$. Suppose that
there exists a sequence of solutions $\{u_\e\}$ of (\ref{maineq}) such that
$$\lim_{\e\to0}\Big(\inf_K|u_\e|\Big)=0.$$
Then, we have $\|u_\e\|_{L^\infty(K)}\to0$ as $\e\to0$.
\end{lemma}
\begin{proof}
Choose a sequence of points $\{x_\e\}\subseteq K$ such that
$|u_\e(x_\e)|=\inf_K|u_\e|$.
Passing to a subsequence (still denoted by $u_\e$),
we may assume that $\lim_{\e\to0}x_\e=x_0\in K$.
We argue by contradiction.
Suppose that there exists a positive constant $c_K>0$ and a sequence $\{z_\e\}\subseteq K$ such that
$\sup_K|u_\e|=|u_\e(z_\e)|\ge c_K$ for small $\e>0$.
We will use the constant $M_0\ge0$ obtained in Lemma \ref{bddofintegration}.
If $u_\e(z_\e)\le-c_K$ then, by using Lemma \ref{propertyofentiresolution}, we can choose $s_1<0$ such that
$$\beta(s_1)>\frac{M_0}{\pi}\ \textrm{and}\ -c_K<s_1<0.$$
If $u_\e(z_\e)\ge c_K$ then, by using Lemma \ref{propertyofentiresolution}, we can choose $s_1>0$ such that
$$\beta(s_1)<-\frac{M_0}{\pi}\ \textrm{and}\ 0<s_1<c_K.$$
We can also choose $y_\e\in K$ such that $u_\e(y_\e)=s_1$ by the intermediate value theorem.
Let $\bar{u}_\e(x)=u_\e(\e x+y_\e)$ for $x\in\Omega_{\e,y_\e}\equiv\{\ x\in\RN\ |\ \e x+y_\e\in K_1\ \}$ where $K_1$ is a compact subset such that
$K\subset\textrm{int}(K_1)\subset\Omega\setminus Z$.
Then $\bar{u}_\e$ satisfies
\begin{equation}
\begin{aligned}
\left\{
 \begin{array}{ll}
 \Delta \bar{u}_\e+\frac{e^{\bar{u}_\e}(1-e^{\bar{u}_\e})}{(\tau+e^{\bar{u}_\e})^3}=0\ \textrm{on}\ \Omega_{\e,y_\e},
 \\ \bar{u}_\e(0)=s_1,
 \\ \int_{\Omega_{\e,y_\e}}\Big|\frac{e^{\bar{u}_\e}(1-e^{\bar{u}_\e})}{(\tau+e^{\bar{u}_\e})^3}\Big| dx\le M_0.
 \end{array}\right.
\end{aligned}
\end{equation}
By using Lemma \ref{green}, we see that $\bar{u}_\e$ is bounded in $C^0_{\textrm{loc}}(\Omega_{\e,y_\e})$.
Passing to a subsequence, we may assume that $\bar{u}_\e$ converges in
$C^2_{\textrm{loc}}(\RN)$ to a function $u_*$ which is a solution of
\begin{equation}
\begin{aligned}
\left\{
 \begin{array}{ll}
 \Delta u_*+\frac{e^{u_*}(1-e^{u_*})}{(\tau+e^{u_*})^3}=0\ \textrm{on}\ \RN,
 \\ u_*(0)=s_1,
 \\ \intr\Big|\frac{e^{u_*}(1-e^{u_*})}{(\tau+e^{u_*})^3}\Big|dx\le M_0.
 \end{array}\right.
\end{aligned}
\end{equation}
By using  Lemma \ref{green} and Lemma \ref{radiallysymm}, we conclude that $u_*$ is
radially symmetric with respect to some point $\bar{p}$ in $\RN$ and $u_*$ does not change sign.
Hence Lemma \ref{propertyofentiresolution} shows that
\begin{equation}
\begin{aligned}
M_0
\ge\Big|\int_{\RN}\frac{e^{u_*}(1-e^{u_*})}{(\tau+e^{u_*})^3}dx\Big|
=2\pi|\beta(u_*(\bar{p}))|\ge2\pi|\beta(s_1)|
>2M_0,
\end{aligned}
\end{equation}
which is the desired contradiction. Therefore, $\lim_{\e\to0}\|u_\e\|_{L^\infty(K)}=0$.
\end{proof}
As a corollary of Lemma \ref{uniformestimates}, we obtain the following proposition.
\begin{prop}\label{alternatives}
Let $u_\e$ be a sequence of solutions of (\ref{maineq}). Then, up to subsequences, one of the following holds true:

(a) $u_\e\to0$ uniformly on any compact subset of $\Omega\setminus Z$;

(b) for any compact subset $K\subset\Omega\setminus Z$, there exists $\nu_K>0$ such that $$\lim_{\e\to0}\Big(\sup_K u_\e\Big)\le-\nu_K;$$

(c) for any compact subset $K\subset\Omega\setminus Z$, there exists $\nu_K>0$ such that $$\lim_{\e\to0}\Big(\inf_K u_\e\Big)\ge\nu_K.$$
\end{prop}
\begin{proof}
In view of Lemma \ref{uniformestimates},
it suffices to show that (a) holds whenever both (b) and (c) fail to hold.
Suppose that (b) and (c) do not hold.
Then, we can take compact sets $K_1,K_2\subset\Omega\setminus Z$ and sequences $\{x_{1,\varepsilon}\}\subset K_1$, $\{x_{2,\varepsilon}\}\subset K_2$ such that
\begin{equation*}
 \lim_{\varepsilon\to 0}u_\varepsilon(x_{1,\varepsilon})\geq 0 \quad {\rm and} \quad
 \lim_{\varepsilon\to 0}u_\varepsilon(x_{2,\varepsilon})\leq 0.
\end{equation*}
For any compact set $K\subset\Omega\setminus Z$, taking a connected compact set $\tilde{K}\subset\Omega\setminus Z$ such that
\begin{equation*}
 \tilde{K}\supseteq K \cup K_1 \cup K_2,
\end{equation*}
and using the intermediate value theorem, we can obtain a sequence $\{x_\varepsilon\}\subseteq \tilde{K}$ satisfying
\begin{equation*}
 \lim_{\varepsilon\to 0}|u_\varepsilon(x_\varepsilon)|=0.
\end{equation*}
Hence, Lemma \ref{uniformestimates} yields that $\lim_{\varepsilon\to 0}\|u_\varepsilon\|_{L^\infty(\tilde{K})}=0$, which completes the proof.
\end{proof}

\noindent
{\bf The proof of Theorem \ref{BrezisMerletypealternatives} completed.}\\
First of all, we assume that (b) in Proposition \ref{alternatives} holds.
In this case, we also suppose that there exists $r\in(0,\frac{1}{3}\textrm{dist}(Z_1,Z_2))$ such that $B_{2r}(p_{i,1})\cap B_{2r}(p_{j,1})=\emptyset$ when $i\neq j$ and
$\lim_{\e\to0}\Big(\sup_{\cup_{j=1}^{d_1}B_r(p_{j,1})}u_\e\Big)\ge0$. By using
$\lim_{x\to p_{j,1}}u_\e(x)=-\infty$ and the intermediate value theorem,
we see that there exists $x_\e\in\overline{\cup_{j=1}^{d_1}B_r(p_{j,1})}$ such that $|u_\e(x_\e)|=\inf_{\cup_{j=1}^{d_1}B_r(p_{j,1})}|u_\e|\to0$ as $\e\to0$. Let
$x_0\in\overline{\cup_{j=1}^{d_1}B_r(p_{j,1})}$ be the limit point of $x_\e$.
Passing to a subsequence, only one of the following two possibilities can be satisfied: either
$x_0\notin Z_1$ or $x_0\in Z_1$.

Case 1: $x_0\notin Z_1$.

Let us fix a  constant $d\in(0,\frac{1}{3}\textrm{dist}(x_0,Z))$. Since $\overline{B_d(x_0)}\subset\Omega\setminus Z$
and in particular
$\lim_{\e\to0}\Big(\inf_{\overline{B_d(x_0)}}|u_\e|\Big)=0$, then, in view of Lemma \ref{uniformestimates},
we see that $\lim_{\e\to0}\Big(\sup_{\overline{B_d(x_0)}}|u_\e|\Big)=0$.
This is a contradiction since we are assuming that Proposition \ref{alternatives} (b) holds.

Case 2: $x_0\in Z_1$.

For the sake of simplicity, we assume that $x_0=0\in Z_1$.
Since we are assuming that  Proposition \ref{alternatives} (b) holds, then
there exists $\gamma>0$ such that $\lim_{\e\to0}\Big(\sup_{|x|=r}u_\e\Big)<-\gamma$.
By the maximum principle, we see that $\sup_{|x|\le r}u_\e\le0$. We claim that
\begin{equation}\label{limxee}
\lim_{\e\to0}\frac{|x_\e|}{\e}=\infty.
\end{equation}
We argue by contradiction and suppose that $\liminf_{\e\to0}\frac{|x_\e|}{\e}<\infty$. Hence,
passing to a subsequence, we could assume that $\frac{|x_\e|}{\e}\le c$ for some constant $c>0$ and small $\e>0$.
Note that $u_\e(x)=2m_{j,1}\ln|x|+v_\e(x)$ near $x=0$ for some smooth function $v_\e$ and $1\le j\le d_1$.
Let $\hat{v}_\e(x)=v_\e(|x_\e|x)+2m_{j,1}\ln|x_\e|$ for $|x|<\frac{r}{|x_\e|}$. Then $\hat{v}_\e$ satisfies
\begin{equation}
\Delta \hat{v}_\e+\frac{|x_\e|^2}{\e^2}\frac{|x|^{2m_{j,1}}e^{\hat{v}_\e}(1-|x|^{2m_{j,1}}e^{\hat{v}_\e})}{(\tau+|x|^{2m_{j,1}}e^{\hat{v}_\e})^3}=0\ \textrm{on}\ B_{\frac{r}{|x_\e|}}(0).
\end{equation}
We also observe that
\begin{equation}\label{uppharnark}
\hat{v}_\e(x)=u_\e(|x_\e|x)-2m_{j,1}\ln|x|\le-2m_{j,1}\ln|x|\ \textrm{for}\ |x|\le \frac{r}{|x_\e|},
\end{equation}
and
\begin{equation}\label{lowharnk}\lim_{\e\to0}\hat{v}_\e\Big(\frac{x_\e}{|x_\e|}\Big)=\lim_{\e\to0}u_\e(x_\e)=0.
\end{equation}
Since $\frac{|x_\e|}{\e}\le c$ and $\sup_{t\ge0}\Big|\frac{t(1-t)}{(\tau+t)^3}\Big|<\infty$, then for any $p>1$ and $R>0$, there exists a constant $C_{p,R}>0$ such that
$\lim_{\e\to0}\|\Delta \hat{v}_\e\|_{L^p(B_R(0))}\le C_{p,R}$. By using
(\ref{uppharnark}), (\ref{lowharnk}), and  Lemma \ref{harnarkineq}, we see that for large $R>0$, there exist
$\sigma\in(0,1)$ and $\gamma>0$, independent of $\e>0$, such that
$$o(1)=\hat{v}_\e\Big(\frac{x_\e}{|x_\e|}\Big)\le\sup_{B_{R/2}(0)} \hat{v}_\e\le\sigma\inf_{B_{R/2}(0)}
\hat{v}_\e+(1+\sigma)\gamma\|\Delta \hat{v}_\e\|_{L^p(B_R(0))}-(1-\sigma)2m_{j,1}\ln R.$$
Hence $\hat{v}_\e$
is bounded in $C^0_{\textrm{loc}}(B_{\frac{r}{|x_\e|}}(0))$. Passing to a subsequence, we may assume that   $\lim_{\e\to0}\frac{x_\e}{|x_\e|}=y_0\in S^1$,
$\lim_{\e\to0}\frac{|x_\e|}{\e}=c_0\ge0$, and $\hat{v}_\e$ converges in $C^2_{\textrm{loc}}(\RN)$ to a function $\hat{v}$ satisfying
\begin{equation}
\Delta \hat{v}+\frac{c_0^2|x|^{2m_{j,1}}e^{\hat{v}}(1-|x|^{2m_{j,1}}e^{\hat{v}})}{(\tau+|x|^{2m_{j,1}}e^{\hat{v}})^3}=0\ \textrm{in}\ \RN.
\end{equation}
Then the function $\hat{u}=\hat{v}+2m_{j,1}\ln|x|\le0$ satisfies
\begin{equation}
\Delta \hat{u}+\frac{c_0^2e^{\hat{u}}(1-e^{\hat{u}})}{(\tau+e^{\hat{u}})^3}=4\pi m_{j,1}\delta_0\ \textrm{in}\ \RN.
\end{equation}
Since $\hat{u}\le0$, we have $c_0>0$ and since $\hat{u}(y_0)=\lim_{\e\to0}u_\e(x_\e)=0$,
we have $\hat{u}\equiv0$ by the strong maximum principle.
This is of course a contradiction and (\ref{limxee}) is proved.

At this point, let us fix a constant $s_2<0$ such that $\beta(s_2)\ge \frac{M_0}{\pi}$ (see (\ref{defbetas}) and Lemma \ref{bddofintegration}) and $-\gamma<s_2<0$. We can
choose $y_\e$ on a line segment joining $x_\e$ to $\frac{rx_\e}{|x_\e|}$ such that
$u_\e(y_\e)=s_2$ and $|y_\e|\ge|x_\e|$ by the intermediate value theorem.
Let $\hat{u}_\e(x)=u_\e(\e x+y_\e)$ on $B_{\frac{|x_\e|}{2\e}}(0)$.
We note that $0\notin B_{\frac{|x_\e|}{2}}(y_\e)$. Then $\hat{u}_\e$ satisfies
\begin{equation}
\begin{aligned}
\left\{
 \begin{array}{ll}
 \Delta \hat{u}_\e+\frac{e^{\hat{u}_\e}(1-e^{\hat{u}_\e})}{(\tau+e^{\hat{u}_\e})^3}=0\ \textrm{in}\ B_{\frac{|x_\e|}{2\e}}(0),
 \\ \hat{u}_\e(0)=s_2,
  \\ \int_{B_{\frac{|x_\e|}{2\e}}(0)}\Big|\frac{e^{\hat{u}_\e}(1-e^{\hat{u}_\e})}{(\tau+e^{\hat{u}_\e})^3}\Big| dx\le M_0.
 \end{array}\right.
\end{aligned}
\end{equation}
By using the fact that $\hat{u}_\e\le0$ and $\hat{u}_\e(0)=s_2$ together with Lemma \ref{harnarkineq}, then we see that
for large $R>0$ there exist $\sigma\in(0,1)$ and $\gamma>0$, independent of $\e>0$, such that
$$s_2=\hat{u}_\e(0)\le\sup_{B_{R/2}(0)} \hat{u}_\e\le\sigma\inf_{B_{R/2}(0)}\hat{u}_\e+(1+\sigma)\gamma\|\Delta \hat{u}_\e\|_{L^p(B_R(0))},$$
and  $\hat{u}_\e$ is bounded in $C^0_{\textrm{loc}}(B_{\frac{|x_\e|}{2\e}}(0))$.
Then $\hat{u}_\e$ converges in $C^2_{\textrm{loc}}(\RN)$ to a function $u_*$ satisfying
\begin{equation}
\begin{aligned}
\left\{
 \begin{array}{ll}
 \Delta u_*+\frac{e^{u_*}(1-e^{u_*})}{(\tau+e^{u_*})^3}=0\ \textrm{in}\ \RN,
 \\u_*(0)=s_2,\ \ u_*\le0,
 \\ \intr\Big|\frac{e^{u_*}(1-e^{u_*})}{(\tau+e^{u_*})^3}\Big| dx\le M_0.
 \end{array}\right.
\end{aligned}
\end{equation}
By using Lemma \ref{radiallysymm}, we see that  $u_*$ is radially symmetric about some point.
Then, we see that  $ \Big|\intr\frac{e^{u_*}(1-e^{u_*})}{(\tau+e^{u_*})^3}\Big| dx\ge 2\pi\beta(s_2)\ge 2M_0$ from
Lemma \ref{propertyofentiresolution} which is once more a contradiction.

At this point, by using the above results, we see that
$\lim_{\e\to0}\Big(\sup_{\cup_{j=1}^{d_1}B_r(p_{j,1})}u_\e\Big)<-c$ for some constant $c>0$, which shows that
(b) in Theorem \ref{BrezisMerletypealternatives} holds whenever (b) in Proposition \ref{alternatives} holds.\\
The proof of (c) in Theorem \ref{BrezisMerletypealternatives} follows essentially by the same argument and we skip it
here to avoid repetitions.  \hspace{\fill}$\square$

\section{Proof of Theorem \ref{stable}: stable solution $\Rightarrow$ topological solution}\label{thmst}
In this section, we will prove one of the implications in the statement of
Theorem \ref{stable}, that is,  stable solution $\Rightarrow$ topological solution  whenever (H1-2) hold.
Let $u_\e$ be a sequence of stable solutions of (\ref{maineq}).
To prove Theorem \ref{stable}, we argue by contradiction and suppose that $u_\e$ does not converge to
$0$ almost everywhere.
Then either $(b)$ or $(c)$ of  Theorem \ref{BrezisMerletypealternatives} would occur.
Since $f_\tau(u)=-\frac{f_{\tau^{-1}}(-u)}{\tau^3}$, without loss of generality
we can assume that $u_\e$ has the profile $(b)$ of Theorem
\ref{BrezisMerletypealternatives}.

If $u_\e-2\ln\e$ has a
bubble at some point in $\Omega\setminus Z_2$, then there are two possibilities. One is that
the limiting equation is the mean field equation and it is easy to see the solution is not stable. Another one is that
the limiting equation is (\ref{maineq}), but defined in the whole $\RN$, and after a suitable scaling,
$u_\e$ tends to a nontopological solution $u$ such that $\lim_{|x|\to\infty}u(x)=-\infty$.
Again, this is also unstable. The proof is not difficult. But for the sake of completeness, we put the proof in the Appendix. To the best of our knowledge,
even for CSH (\ref{chseq}), this result has not been written in the literature.

 Therefore, from now on,  we may assume
that for any small $r>0$,  there exists $c_r>0$ such that
\begin{equation}\label{oriassm}
w_\e\equiv u_\e-2\ln\e<c_r\ \ \textrm{on}\ \ \Omega\setminus\cup_{j}(B_r(p_{j,2})).
\end{equation}

Now we consider
\begin{equation}
\begin{aligned}\label{nextforeigenvalue}
&\mu_\e\equiv\inf_{\phi\in W^{1,2}(\Omega)\setminus\{0\}}\frac{\int_{\Omega}|\nabla\phi|^2-\frac{1}{\e^2}f'_\tau(u_\e)\phi^2dx}{\|\phi\|^2_{L^2(\Omega)}}
\\&\le\frac{1}{|\Omega|^2}\int_\Omega -\frac{1}{\e^2}f'_\tau(u_\e)dx
=\frac{1}{|\Omega|^2}\int_\Omega \frac{e^{u_\e}(-\tau+2(\tau+1)e^{u_\e}-e^{2u_\e})}{\e^2(\tau+e^{u_\e})^4}dx.
\end{aligned}
\end{equation}
To derive a contradiction, we want to prove that for small $\e>0$,
\begin{equation}\label{eigen}
\begin{aligned}\int_\Omega \frac{e^{u_\e}(-\tau+2(\tau+1)e^{u_\e}-e^{2u_\e})}{\e^2(\tau+e^{u_\e})^4}dx
<0.\end{aligned}
\end{equation}
To prove (\ref{eigen}), we need to compute the integral over a small neighborhood of each $p_{j,2}\in Z_2$.
Let us first show a simple fact about $w_\e$.

\begin{lemma}\label{twocases}$w_\e$ satisfies
\begin{equation}\label{nablabdd1}
\textrm{either}\ \ \lim_{\e\to0}\|w_\e-u_0\|_{L^\infty(\Omega\setminus\cup_{j}(B_r(p_{j,2})))}<
\infty\ \ \textrm{or}\ \ \lim_{\e\to0}\Big(\sup_{\Omega\setminus\cup_{j}(B_r(p_{j,2}))}w_\e\Big)=-\infty.\end{equation}
Moreover, for any small $r>0$, there exists $C_r>0$ such that
\begin{equation}\label{nablabdd}
\sup_{\Omega\setminus\cup_{j}(B_r(p_{j,2}))}|\nabla (w_\e-u_0)|\le C_r.
\end{equation}
\end{lemma}
\begin{proof}
We note that $w_\e$ satisfies the following equation
\begin{equation}\label{eqforwe}
\Delta w_\e+\frac{e^{w_\e}(1-\e^2e^{w_\e})}{(\tau+\e^2e^{w_\e})^3}=
4\pi\sum^{d_{1}}_{j=1}m_{j,1}\delta_{p_{j,1}}-4\pi\sum^{d_{2}}_{j=1}m_{j,2}\delta_{p_{j,2}}
\ \textrm{on}\ \Omega.
\end{equation}
We also see that
$$\Delta (w_\e-u_0)+\frac{e^{w_\e}(1-\e^2e^{w_\e})}{(\tau+\e^2e^{w_\e})^3}=\frac{4\pi(N_1-N_2)}{|\Omega|}
\ \textrm{on}\ \Omega.$$
By using (\ref{oriassm}) and Lemma \ref{harnarkineq}, we readily obtain (\ref{nablabdd1}).

Next, by using the Green's representation formula for a solution $w_\e$ of (\ref{eqforwe}), we see that for $x\in \Omega$,
\begin{equation}
\begin{aligned}
w_\e(x)-u_0(x)
=\frac{1}{|\Omega|}\int_\Omega w_\e(y)dy+\int_\Omega G(x,y) \frac{e^{w_\e}(1-\e^2e^{w_\e})}{(\tau+e^{w_\e})^3}dy.
\end{aligned}
\end{equation}
By using Lemma \ref{bddofintegration}, we conclude that there exists a constant $C>0$,
independent of $\e>0$ and $r>0$, such that for $x\in\Omega\setminus\cup_{j}(B_{r}(p_{j,2}))$, it holds
\begin{equation}
\begin{aligned}
|\nabla (w_\e(x)-u_0(x))|
&\le \frac{1}{2\pi}\int_\Omega\frac{1}{|x-y|}\Big|\frac{e^{w_\e}(1-\e^2e^{w_\e})}{(\tau+\e^2e^{w_\e})^3}\Big|dy+C
\\&\le \frac{1}{2\pi}\Big\{\sup_{\Omega\setminus\cup_{j}(B_{\frac{r}{2}}(p_{j,2}))}\Big|\frac{e^{w_\e}(1-\e^2e^{w_\e})}{(\tau+\e^2e^{w_\e})^3}\Big|\int_{B_{\frac{r}{2}}(x)}\frac{1}{|x-y|}dy
\\&+\frac{2}{r}\int_{\Omega\setminus\ B_{\frac{r}{2}}(x)}\Big|\frac{e^{w_\e}(1-\e^2e^{w_\e})}{(\tau+\e^2e^{w_\e})^3}\Big|dy\Big\}+C.
\end{aligned}
\end{equation}
By using (\ref{oriassm}) and Lemma \ref{bddofintegration}, we obtain (\ref{nablabdd}) which concludes
the proof of our lemma.
\end{proof}

If $\lim_{\e\to0}\|w_\e-u_0\|_{L^\infty(\Omega\setminus\cup_{j}(B_r(p_{j,2})))}<\infty\ \ \textrm{for any small }\ \ r>0$, then there exists a function $w$ satisfying
\begin{equation*}
w_\e\to w\ \ \textrm{in}\ \ C^2_{\textrm{loc}}(\Omega\setminus Z_2).
\end{equation*}
By using Lemma \ref{bddofintegration}, we also see that $w$ satisfies
\begin{equation}\label{betajbddminusone}
\Delta w+\frac{e^w}{\tau^3}=4\pi\sum^{d_{1}}_{j=1}m_{j,1}\delta_{p_{j,1}}+4\pi\sum^{d_{2}}_{j=1}\beta_{j,2}\delta_{p_{j,2}}
\ \textrm{on}\ \Omega\ \ \textrm{where }\ \ \beta_{j,2}>-1.
\end{equation}

If $\lim_{\e\to0}\Big(\sup_{\Omega\setminus\cup_{j}(B_r(p_{j,2}))}w_\e\Big)=-\infty$,
then for fixed  $x_0\in\Omega\setminus Z$, and by using (\ref{nablabdd}), we see that
there exists a function $g$ satisfying
\begin{equation*}
g_\e\equiv w_\e-w_\e(x_0)\to g\ \ \textrm{in}\ \ C^2_{\textrm{loc}}(\Omega\setminus Z_2),
\end{equation*} and
\begin{equation}\label{inftyftn}\Delta g=4\pi\sum^{d_{1}}_{j=1}m_{j,1}\delta_{p_{j,1}}+4\pi\sum^{d_{2}}_{j=1}\beta_{j,2}\delta_{p_{j,2}}
\ \textrm{on}\ \Omega \ \ \textrm{where }\ \ \beta_{j,2}\in\R.
\end{equation}
Clearly (\ref{inftyftn}) implies $N_1+\sum_{j=1}^{d_2}\beta_{j,2}=0$.

Next we have the following property.
\begin{lemma}\label{intfandcapf}For any $1\le j\le d_2$,
\begin{equation}\label{poeqforftwotau1}
\lim_{r\to0}\lim_{\e\to0}\int_{B_r(p_{j,2})}\frac{f_\tau(u_\e)}{\e^2}dx
=-4\pi(m_{j,2}+\beta_{j,2}),
\end{equation}
and
\begin{equation}\begin{aligned}\label{poeqforftwotau}
&\lim_{r\to0}\lim_{\e\to0}\int_{B_r(p_{j,2})}\frac{F_{2,\tau}(u_\e)}{\e^2}dx
=2\pi(\beta_{j,2}^2- m_{j,2}^2),
\end{aligned}
\end{equation}
where $F_{2,\tau}(u)\equiv\frac{e^u((1-\tau)e^u+2\tau)}{2\tau^2(\tau+e^u)^2}$.
\end{lemma}
\begin{proof}
For the sake of simplicity, we assume that $p_{j,2}=0$. We consider the following two cases.

\textit{Case 1.} $w_\e\to w\ \ \textrm{in}\ \ C^2_{\textrm{loc}}(\Omega\setminus Z_2).$\\
We integrate  (\ref{eqforwe}) on $B_r(0)$ and take the limit as $\e\to0$ to conclude that
\begin{equation*}
\begin{aligned}
\lim_{\e\to0}\int_{B_r(0)}\frac{f_\tau(u_\e)}{\e^2}dx&=-\lim_{\e\to0}\Big(4\pi m_{j,2}+\int_{\partial B_r(0)}\frac{\partial w_\e}{\partial\nu}d\sigma\Big)
=-\Big(4\pi m_{j,2}+\int_{\partial B_r(0)}\frac{\partial w}{\partial\nu}d\sigma\Big)
\\&=-4\pi(m_{j,2}+\beta_{j,2})+\int_{B_r(0)}\frac{e^w}{\tau^3}dx.
\end{aligned}
\end{equation*}
Clearly Lemma \ref{bddofintegration} implies that
\begin{equation*}
\lim_{r\to0}\lim_{\e\to0}\int_{B_r(0)}\frac{f_\tau(u_\e)}{\e^2}dx
=-4\pi(m_{j,2}+\beta_{j,2}).
\end{equation*}

At this point we  consider the function $v\equiv w-2\beta_{j,2}\ln|x|$ which satisfies
\begin{equation}\label{eqforv}
\Delta v+\frac{e^w}{\tau^3}=0
\ \textrm{on}\ B_r(0).
\end{equation}
Multiplying  (\ref{eqforv}) by $\nabla w\cdot x$ and integrating over $B_r(0)$,  we conclude that
\begin{equation}\label{pov}
\begin{aligned}
&\int_{\partial B_r(0)}\Big[\Big(\nabla v\cdot\frac{x}{|x|}\Big)(\nabla v\cdot x)-\frac{|\nabla v|^2|x|}{2}+\frac{e^{w}|x|}{\tau^3}\Big]d\sigma
=\int_{B_r(0)}\frac{(2+2\beta_{j,2})e^{w}}{\tau^3}dx.
\end{aligned}
\end{equation}
Let us also consider the function $v_\e(x)\equiv u_\e(x)+2m_{j,2}\ln|x|$ which  satisfies
\begin{equation}\label{eqforve}
\Delta v_\e+\frac{f_\tau(u_\e)}{\e^2}=0
\ \textrm{on}\ B_r(0).
\end{equation}
Multiplying  (\ref{eqforve}) by $\nabla u_\e\cdot x$ and integrating over $B_r(0)$, we have
\begin{equation}
\begin{aligned}\label{fornextpo2}
&\int_{B_r(0)}\frac{2F_{2,\tau}(u_\e)}{\e^2}dx
\\&=\int_{\partial B_r(0)}\Big[\Big(\nabla v_\e\cdot\frac{x}{|x|}\Big)(\nabla v_\e\cdot x)-\frac{|\nabla v_\e|^2|x|}{2}+\frac{F_{2,\tau}(u_\e)|x|}{\e^2}-2m_{j,2}\frac{\nabla v_\e\cdot x}{|x|}\Big]d\sigma.
\end{aligned}
\end{equation}
Hence, as $\e\to0$, we have
\begin{equation*}
\begin{aligned}
&\lim_{\e\to0}\int_{B_r(0)}\frac{2F_{2,\tau}(u_\e)}{\e^2}dx
\\&=\int_{\partial B_r(0)}\Big[\frac{(\nabla v\cdot x+2(m_{j,2}+\beta_{j,2}))^2}{|x|}-\Big|\nabla v+\frac{2(m_{j,2}+\beta_{j,2})x}{|x|^2}\Big|^2\frac{|x|}{2}
\\&+\frac{e^{w}|x|}{\tau^3}
-\frac{2m_{j,2}\{\nabla v\cdot x+2(m_{j,2}+\beta_{j,2})\}}{|x|}\Big]d\sigma.
\end{aligned}
\end{equation*}
By using (\ref{pov}), we also see that
\begin{equation*}\begin{aligned}
&\lim_{r\to0}\lim_{\e\to0}\int_{B_r(0)}\frac{F_{2,\tau}(u_\e)}{\e^2}dx
=2\pi(\beta_{j,2}^2- m_{j,2}^2),
\end{aligned}
\end{equation*}
which is (\ref{poeqforftwotau}).

\

\textit{Case 2.} $g_\e=w_\e-w_\e(x_0)\to g\ \ \textrm{in}\ \ C^2_{\textrm{loc}}(\Omega\setminus Z_2).$\\
We integrate  (\ref{eqforwe}) on $B_r(0)$ and take the limit as $\e\to0$ to conclude that
\begin{equation*}
\begin{aligned}
\lim_{\e\to0}\int_{B_r(0)}\frac{f_\tau(u_\e)}{\e^2}dx&=-\lim_{\e\to0}\Big(4\pi m_{j,2}+\int_{\partial B_r(0)}\frac{\partial g_\e}{\partial\nu}d\sigma\Big)
\\&=-\Big(4\pi m_{j,2}+\int_{\partial B_r(0)}\frac{\partial g}{\partial\nu}d\sigma\Big)
=-4\pi(m_{j,2}+\beta_{j,2}).
\end{aligned}
\end{equation*}

Let us consider the function $h\equiv g-2\beta_{j,2}\ln|x|$. Then $h$ satisfies
\begin{equation}\label{pov2}
\Delta h=0
\ \textrm{on}\ B_r(0).
\end{equation}
Next we also define $h_\e(x)\equiv g_\e(x)+2m_{j,2}\ln|x|$ which satisfies
\begin{equation}\label{eqforvehe}
\Delta h_\e+\frac{f_\tau(u_\e)}{\e^2}=0
\ \textrm{on}\ B_r(0).
\end{equation}
Multiplying  (\ref{eqforvehe}) by $\nabla u_\e\cdot x$ and integrating over $B_r(0)$,  we see that
\begin{equation*}
\begin{aligned}
&\lim_{\e\to0}\int_{B_r(0)}\frac{2F_{2,\tau}(u_\e)}{\e^2}dx
\\&=\int_{\partial B_r(0)}\Big[\frac{(\nabla h\cdot x+2(m_{j,2}+\beta_{j,2}))^2}{|x|}-\Big|\nabla h+\frac{2(m_{j,2}+\beta_{j,2})x}{|x|^2}\Big|^2\frac{|x|}{2}
\\&-\frac{2m_{j,2}\{\nabla h\cdot x+2(m_{j,2}+\beta_{j,2})\}}{|x|}\Big]d\sigma.
\end{aligned}
\end{equation*}
By using (\ref{pov2}), we also conclude that
\begin{equation*}\begin{aligned}
&\lim_{\e\to0}\int_{B_r(0)}\frac{F_{2,\tau}(u_\e)}{\e^2}dx
=2\pi(\beta_{j,2}^2- m_{j,2}^2),
\end{aligned}
\end{equation*}
which is (\ref{poeqforftwotau1}).
\end{proof}

Let $\hat{u}_\e(x)=u_\e(\e x)$ which satisfies
$$\Delta\hat{u}_\e+f_\tau(\hat{u}_\e)=-4\pi m_{j,2}\delta_{p_{j,2}}\ \ \textrm{on}  \ \ B_{\frac{r}{\e}}(p_{j,2}).$$
Moreover we have:
\begin{lemma} There exists a constant $c>0$, independent of $r>0$ and $\e>0$, such that  \begin{equation}\label{gradgrad}
\Big|\nabla \hat{u}_\e(x)+\frac{2m_{j,2}(x-p_{j,2})}{|x-p_{j,2}|^2}\Big|\le c \ \ \textrm{on}\ \ B_{\frac{r}{\e}}(p_{j,2}).
\end{equation}
\end{lemma}
\begin{proof} For the sake of simplicity, we assume that $p_{j,2}=0$.
  By using the Green's representation formula for a solution $u_\e$ of (\ref{maineq}) (see (\ref{greenrp})) and Lemma \ref{bddofintegration}, we  see that for
 $x\in B_r(0)$,
 \begin{equation*}
\begin{aligned}
\Big|\nabla u_\e(x)+\frac{2m_{j,2}x}{|x|^2}\Big|&\le C+\frac{1}{2\pi\e^2}\int_\Omega \frac{|f_\tau(u_\e)|}{|x-y|}dy
\\&=C+\frac{1}{2\pi\e^2}\Big(\int_{B_\e(x)} \frac{|f_\tau(u_\e)|}{|x-y|}dy+\int_{\Omega\setminus B_\e(x)} \frac{|f_\tau(u_\e)|}{|x-y|}dy\Big)
\\&\le C+\frac{C'}{\e},
\end{aligned}
\end{equation*}
for some constants $C,\ C'>0$, independent of $r>0$ and $\e>0$.  The desired conclusion
follows by the substitution $\hat{u}_\e(x)=u_\e(\e x)$.
\end{proof}
As mentioned above, we have to study the behavior of $\hat{u}_\e$ as $\e\to0$. This is most delicate part of our proof. Here, the Pohozaev identity (\ref{poeqforftwotau}) is used.
\begin{lemma}\label{minusinftyintheregion}If  $\tau=1$ or $m_{j,2}\in[0,1]$, then for  any $\eta>0$,
\begin{equation}\label{numberforlemma}
\lim_{\e\to0}\Big(\sup_{|x-p_{j,2}|=\eta}\hat{u}_\e(x)\Big)=\lim_{\e\to0}\Big(\sup_{|x-p_{j,2}|=\eta}u_\e(\e x)\Big)=-\infty.
\end{equation}
Moreover, if  $w_\e\to w\ \ \textrm{in}\ \ C^2_{\textrm{loc}}(\Omega\setminus Z_2)$,
then (\ref{numberforlemma}) always holds without any further assumptions for $\tau$ and  $m_{j,2}$.
\end{lemma}
\begin{proof}For the sake of simplicity, we assume that $p_{j,2}=0$. We divide the proof of our
lemma in two steps.

\textit{Step 1.} We claim that for any $\eta>0$, there exists $c_\eta>0$ such that for small $\e>0$,
\begin{equation}\label{prove}
\sup_{|x|=\eta}\hat{u}_\e(x)=\sup_{|x|=\eta}u_\e(\e x)<c_\eta.
\end{equation}
We argue by contradiction and suppose that
there exists $\eta_0>0$ such that
\begin{equation*}
\lim_{\e\to0}\Big(\sup_{|x|=\eta_0}\hat{u}_\e(x)\Big)=\lim_{\e\to0}\Big(\sup_{|x|=\eta_0}u_\e(\e x)\Big)=
+\infty.
\end{equation*}
Since $|\nabla \hat{u}_\e|$ is locally bounded in $B_{\frac{r}{\e}}(0)\setminus\{0\}$, we also  see that
\begin{equation}\label{supp}
\lim_{\e\to0}\Big(\inf_{|x|=\eta_0}\hat{u}_\e(x)\Big)=\lim_{\e\to0}\Big(\inf_{|x|=\eta_0}u_\e(\e x)\Big)=
+\infty.
\end{equation}
Fix $c\in(-\infty,0)$ and $n\in\mathbb{N}$. Since $\lim_{\e\to0}\sup_{\partial B(0,r/\e)}\hat{u}_\e=-\infty$,
then (\ref{supp}) implies that there exists $y_\e^i=(r_\e^i\cos\theta_i,r_\e^i\sin\theta_i) $ such that
$\hat{u}_\e(y_\e^i)=c$ where $\theta_i=\frac{2\pi i}{n}$ and
 $$\lim_{\e\to0}r_\e^i=+\infty,  \ \ \lim_{\e\to0}(\e r_\e^i)=0\ \ \textrm{for all}\ \ 1\le i\le n.$$
In view of (\ref{gradgrad}), we see that  the function $\bar{u}_\e^i(x)=\hat{u}_\e(x+y_\e^i)$ satisfies
$$\Delta\bar{u}_\e^i+f_\tau(\bar{u}_\e^i)=0, \ |\nabla \bar{u}_\e^i|\le C_1\ \ \textrm{on}  \ \ B_{\frac{r_\e^i}{2}}(0),\ \ \bar{u}_\e^i(0)=c<0,$$
for some constant $C_1>0$.
Then $\{\bar{u}_\e^i\}$ is uniformly bounded in $L^\infty_{\textrm{loc}}(B_{\frac{r_\e^i}{2}}(0))$ and  there exists a function $\bar{u}^i$ such that $\bar{u}_\e^i\to\bar{u}^i$ in
$C^2_{\textrm{loc}}(\RN)$ and $$\Delta\bar{u}^i+f_\tau(\bar{u}^i)=0, \ |\nabla \bar{u}^i|\le C_1\ \ \textrm{on}  \ \ \RN,\ \ \bar{u}^i(0)=c<0.$$
By using Lemma \ref{radiallysymm}, we see that $\bar{u}^i$ is radially symmetric with respect to some point
$\bar{p}^i$ in $\RN$ and $\bar{u}^i$ does not change sign.
Hence Lemma \ref{bddofintegration} and  Lemma \ref{propertyofentiresolution} together imply that
there exists a large $R>0$ such that $\int_{B_R(0)}|f_\tau(\bar{u}^i_\e)|dx\ge 4\pi$. Then,
\begin{equation*}
\begin{aligned}
M_0&\ge\int_{B_{\frac{r}{\e}}(0)}|f_\tau(\hat{u}_\e)|dx\ge\sum_{i=1}^n\int_{B_R(y_\e^i)}|f_\tau(\hat{u}_\e)|dx
\\&=\sum_{i=1}^n\int_{B_R(0)}|f_\tau(\bar{u}_\e^i)|dx
\ge 4\pi n\ \ \textrm{for any}\ \ n\in\mathbb{N},
\end{aligned}
\end{equation*} which is a contradiction. Therefore (\ref{prove}) holds as claimed.

Moreover, by using  (\ref{oriassm}), (\ref{prove}) and the maximum principle,
we obtain
 \begin{equation}\label{uppboundd}
\sup_{\eta\le|x|\le\frac{r}{\e}}\hat{u}_\e(x)=\sup_{\eta\le|x|\le\frac{r}{\e}}u_\e(\e x)<c_\eta.
\end{equation}

\textit{Step 2.} To prove our lemma, we argue by contradiction and suppose   that $\{\hat{u}_\e\}$ is uniformly bounded in $L^\infty_{\textrm{loc}}(B_{\frac{r}{\e}}(0)\setminus\{0\})$.
Then, since  $\sup_{t\in\R}|f_{\tau}(t)|<\infty$ and by using (\ref{gradgrad}) and (\ref{uppboundd}), we see that
there exists a function $\hat{u}$ such that $\hat{u}_\e\to\hat{u}$ in
$C^2_{\textrm{loc}}(\RN\setminus\{0\})$ and
\begin{equation}
\begin{aligned}\label{limitlimit}
\left\{
\begin{array}{ll}
\Delta\hat{u}+f_\tau(\hat{u})=-4\pi m_{j,2}\delta_0\ \ \ \textrm{on}  \ \ \RN,
 \\
 \\ f_\tau(\hat{u})\in L^1(\RN),\ \ \sup_{|x|\ge1} \hat{u}(x)\le C,\ \ \sup_{|x|\ge1}|\nabla \hat{u}(x)|\le C,
 \end{array}\right.
\end{aligned}
\end{equation}
for some constant $C>0$.
 Let $\hat{\beta}=\frac{1}{2\pi}\intr f_\tau(\hat{u})dx$.
Then we obtain
$$\lim_{|x|\to\infty}\frac{\hat{u}(x)}{\ln|x|}=-2m_{j,2}-\hat{\beta}.$$
In view of (\ref{limitlimit}), we also see that we cannot have $\lim_{|x|\to\infty}\hat{u}(x)>0$.
Moreover, since $f_\tau(\hat{u})\in L^1(\RN)$ and $\sup_{|x|\ge1}|\nabla \hat{u}(x)|\le C$, then we see that
\begin{equation}\label{eitheror}\textrm{either}\ \ \lim_{|x|\to\infty}\hat{u}(x)=0\ \ \textrm{or} \lim_{|x|\to\infty}\hat{u}(x)=-\infty.
\end{equation}
Indeed, if there exists  a sequence $x_n\in\RN$ such that,
$$\lim_{n\to\infty}|x_n|\to+\infty, \ \ \lim_{n\to\infty}\hat{u}(x_n)=c\notin\{0,\ -\infty\},$$
then since $\sup_{|x|\ge1}|\nabla \hat{u}(x)|\le C$,
there exist small $r_0>0$ and $c_0>0$, independent of $n$, such that
$$|f_\tau(\hat{u})|\ge c_0>0\ \ \textrm{on}\ \ B_{r_0}(x_n).$$
Then $\intr |f_\tau(\hat{u})|dx\ge\sum_{n=1}^\infty\int_{B_{r_0}(x_n)}|f_\tau(\hat{u})|dx=+\infty$
which proves (\ref{eitheror}).

If $\lim_{|x|\to\infty}\hat{u}(x)=0$, then (\ref{oriassm}) and the maximum principle imply that
there exist $c_\tau>0$ and  $R_0>0$ such that
\begin{equation}\label{eq1212}
 \hat{u}_\e<c_\tau\ \ \ \textrm{on}\ \  B_{\frac{r}{\e}}(0)\setminus B_{R_0}(0),
 \end{equation}
 which implies that $$F_{2,\tau}(\hat{u}_\e)>0\ \ \textrm{on}\ \  B_{\frac{r}{\e}}(0)\setminus B_{R_0}(0).$$
In view of $\sup_{t\in\R}|F_{2,\tau}(t)|<\infty$, (\ref{poeqforftwotau}) and  (\ref{eq1212}),
we also see that, for any $R\in(R_0,\infty)$, we have
\begin{equation}
\begin{aligned}\label{contrafrompo}
2\pi(\beta_{j,2}^2- m_{j,2}^2)&=\lim_{r, \eta \to0}\lim_{\e\to0}\int_{B_{\frac{r}{\e}(0)}\setminus B_\eta(0)}F_{2,\tau}(\hat{u}_\e)dx
\\&\ge\lim_{\eta \to0}\lim_{\e\to0}\int_{B_R(0)\setminus B_\eta(0)}F_{2,\tau}(\hat{u}_\e)dx
=\int_{B_R(0)}F_{2,\tau}(\hat{u})dx.
\end{aligned}
\end{equation}
Since $\lim_{|x|\to\infty}\hat{u}(x)=0$,
we see that
$$\lim_{|x|\to\infty}F_{2,\tau}(\hat{u})
=\lim_{|x|\to\infty}\frac{e^{\hat{u}}((1-\tau)e^{\hat{u}}+2\tau)}{2\tau^2(\tau+e^{\hat{u}})^2}=\frac{1}{2(\tau+1)\tau^2}\neq0,$$
which shows that the right hand side of (\ref{contrafrompo}) could be arbitrarily large, which is impossible.
Hence the first case in (\ref{eitheror}) cannot occur.

If $\lim_{|x|\to\infty}\hat{u}(x)=-\infty$, then in view of
(\ref{oriassm}) and the maximum principle, there exists $R_0>0$ such that
\begin{equation}
\begin{aligned}\label{strictlylessinfty}
  \hat{u}_\e<0 \ \ \textrm{on}\ \  B_{\frac{r}{\e}}(0)\setminus B_{R_0}(0).
\end{aligned}
\end{equation}
By using Lemma \ref{intfandcapf} and (\ref{strictlylessinfty}), we see that
\begin{equation*}
\begin{aligned}
2\pi\hat{\beta}&=\lim_{R\to\infty}\int_{|x|\le R}f_\tau(\hat{u})dx
=\lim_{R\to\infty}\lim_{\e\to0}\int_{|x|\le R}f_\tau(\hat{u}_\e)dx
\\&\le\lim_{r\to0}\lim_{\e\to0}\int_{|x|\le \frac{r}{\e}}f_\tau(\hat{u}_\e)dx
= -4\pi(m_{j,2}+\beta_{j,2}).
\end{aligned}
\end{equation*}
Hence we conclude that
\begin{equation*}
\hat{\beta}\le-2(m_{j,2}+\beta_{j,2}),
\end{equation*}
and in particular that
\begin{equation}\label{lnx1}
\lim_{|x|\to\infty}\frac{\hat{u}(x)}{\ln|x|}=-2m_{j,2}-\hat{\beta}\ge2\beta_{j,2}.
\end{equation}
By using (\ref{limitlimit}) and $\lim_{|x|\to\infty}\hat{u}(x)=-\infty$, we see that $e^{\hat{u}}\in L^1(\RN\setminus B_1(0))$ and then
\begin{equation}\label{lnx2}
\lim_{|x|\to\infty}\frac{\hat{u}(x)}{\ln|x|}<-2.
\end{equation}
At this point, the method of moving planes can be used together with (\ref{lnx2}) to prove that
$\hat{u}$ is radially symmetric (see  \cite{ChL, GNN}).
Moreover, (\ref{lnx1}) and (\ref{lnx2}) imply that
\begin{equation}
\label{rangeforbetainfty}\beta_{j,2}<-1.
\end{equation}
If  $w_\e\to w\ \ \textrm{in}\ \ C^2_{\textrm{loc}}(\Omega\setminus Z_2)$,
then (\ref{rangeforbetainfty}) contradicts (\ref{betajbddminusone}).
Moreover, if $\tau=1$ or $m_{j,2}\in[0,1]$, then Theorem 3.4 in \cite{CHLL}
imply that $\hat{u}$ cannot be stable solution, which yields a contradiction and
completes the proof of our lemma.
\end{proof}

\begin{lemma}\label{mplusbetaeq0} If \begin{equation}\label{as1}
\lim_{\e\to0}\Big(\sup_{|x-p_{j,2}|=\eta}\hat{u}_\e(x)\Big)=\lim_{\e\to0}\Big(\sup_{|x-p_{j,2}|=\eta}u_\e(\e x)\Big)=-\infty,
\end{equation} then   \begin{equation}\label{anotherclaim}
\lim_{\e\to0}\Big(\sup_{\eta\le|x-p_{j,2}|\le\frac{r}{\e}}\hat{u}_\e(x)\Big)=\lim_{\e\to0}\Big(\sup_{\eta\le|x-p_{j,2}|\le\frac{r}{\e}}u_\e(\e x)\Big)=-\infty.
\end{equation} Moreover, $m_{j,2}+\beta_{j,2}=0$.
\end{lemma}
\begin{proof}For the sake of simplicity, we assume that $p_{j,2}=0$. We divide the proof of
our lemma in the following steps.

\textit{Step 1.}
 To prove (\ref{anotherclaim}),  we argue by contradiction and suppose that for some constant $c\in(-\infty,0)$, there exists $y_\e\in B_{\frac{r}{\e}}(0)\setminus B_\eta(0)$ such that
$\hat{u}_\e(y_\e)=c$.
In view of (\ref{oriassm}), (\ref{gradgrad}) and (\ref{as1}), we see that $$\lim_{\e\to0}|y_\e|=\infty \ \ \textrm{and} \ \ \ \lim_{\e\to0}(\e|y_\e|)=0.$$
Moreover, by using (\ref{gradgrad}), we see that  the function $\bar{u}_\e(x)=\hat{u}_\e(x+y_\e)$ satisfies
\begin{equation*}
\begin{aligned}
\left\{
\begin{array}{ll}
\Delta\bar{u}_\e+f_\tau(\bar{u}_\e)=0, \ |\nabla \bar{u}_\e|\le C_1\ \ \textrm{on}  \ \ B_{\frac{|y_\e|}{2}}(0),
 \\ \bar{u}_\e(0)=c<0,\ \ f_\tau(\bar{u}_\e)\in L^1(B_{\frac{|y_\e|}{2}}(0)),\
 \end{array}\right.
\end{aligned}
\end{equation*}
for some constant $C_1>0$.
Then $\{\bar{u}_\e\}$ is uniformly bounded in $L^\infty_{\textrm{loc}}(B_{\frac{|y_\e|}{2}}(0))$ and  there exists a function $\bar{u}$ such that $\bar{u}_\e\to\bar{u}$ in
$C^2_{\textrm{loc}}(\RN)$ and
\begin{equation*}
\begin{aligned}
\left\{
\begin{array}{ll}
\Delta\bar{u}+f_\tau(\bar{u})=0, \ |\nabla \bar{u}|\le C_1\ \ \textrm{on}  \ \ \RN,
 \\ \bar{u}(0)=c<0,\ \ f_\tau(\bar{u})\in L^1(\RN).
 \end{array}\right.
\end{aligned}
\end{equation*}
By using   Lemma \ref{radiallysymm}, we conclude that $\bar{u}$ is a nontopological radially symmetric  solution.
Then Theorem 3.4 in \cite{CHLL} shows that $\bar{u}$ cannot be a stable solution which
proves (\ref{anotherclaim}).

\textit{Step 2.} By using Lemma \ref{intfandcapf}  and (\ref{anotherclaim}) we see that
 \begin{equation}\begin{aligned}\label{as3infty}
-4\pi(m_{j,2}+\beta_{j,2})
 &=\lim_{r, \eta\to0}\lim_{\e\to0}\int_{B_{\frac{r}{\e}}(0)\setminus B_{\eta}(0)}f_\tau(\hat{u}_\e)dx
=\lim_{r, \eta\to0}\lim_{\e\to0}\int_{B_{\frac{r}{\e}}(0)\setminus B_{\eta}(0)}\frac{e^{\hat{u}_\e}}{\tau^3}dx
  \\&=\lim_{r, \eta\to0}\lim_{\e\to0}\int_{B_{\frac{r}{\e}}(0)\setminus B_{\eta}(0)}F_{2,\tau}(\hat{u}_\e)dx
  =2\pi(\beta_{j,2}^2-m_{j,2}^2),
 \end{aligned}\end{equation}
 which implies
 \begin{equation*}
\begin{aligned}
\left\{
 \begin{array}{ll}
m_{j,2}+\beta_{j,2}\le0,\ \ \textrm{and}
\\ \textrm{either} \ \ m_{j,2}+\beta_{j,2}=0\ \ \textrm{or}\ \ \ m_{j,2}-\beta_{j,2}=2.
 \end{array}\right.
\end{aligned}
\end{equation*} \begin{equation*}
 \end{equation*}
To prove our lemma, we  argue by contradiction and suppose that  \begin{equation}\label{betageminus11}m_{j,2}+\beta_{j,2}<0,\end{equation}
which implies
 \begin{equation}\label{betageminus1}
m_{j,2}-\beta_{j,2}=2, \ \  \ m_{j,2}<1<-\beta_{j,2}.
\end{equation}

If $w_\e\to w\ \ \textrm{in}\ \ C^2_{\textrm{loc}}(\Omega\setminus Z_2)$, then (\ref{betageminus1}) contradicts $\beta_{j,2}>-1$ in  (\ref{betajbddminusone}), and we obtain that $m_{j,2}+\beta_{j,2}=0$ in this case.

Therefore, from now on,  we assume that
\begin{equation}\label{assuminlemma}\lim_{\e\to0}\Big(\sup_{\Omega\setminus\cup_{j}(B_r(p_{j,2}))}w_\e\Big)=-\infty\ \ \textrm{for any small}\ \ r>0.\end{equation}
Then (\ref{as1}) and (\ref{assuminlemma}) imply that for any $r, \eta>0$,
\begin{equation}\label{claimbounderycond00} \lim_{\e\to0}\Big(\sup_{x\in \partial B_{\frac{r}{\e}}(0)\cup\partial B_\eta(0)}e^{\hat{u}_\e(x)}|x|^2\Big)=0.
\end{equation}

\textit{Step 3.} We claim that for any $r, \eta>0$,
\begin{equation}\label{claimbounderycond2} \lim_{\e\to0}\Big(\sup_{x\in B_{\frac{r}{\e}}(0)\setminus B_\eta(0)}e^{\hat{u}_\e(x)}|x|^2\Big)=0.
\end{equation}
Let us choose $y_\e\in B_{\frac{r}{\e}}(0)\setminus B_\eta(0)$ such that
$$e^{\hat{u}_\e(y_\e)}|y_\e|^2=\Big(\sup_{x\in B_{\frac{r}{\e}}(0)\setminus B_\eta(0)}e^{\hat{u}_\e(x)}|x|^2\Big).$$
We consider the function $\tilde{u}_\e(x)\equiv\hat{u}_\e(|y_\e| x)+2\ln |y_\e|$. Then $\tilde{u}_\e$ satisfies
\begin{equation*}
\Delta \tilde{u}_\e+ \frac{e^{\tilde{u}_\e}(1-e^{\tilde{u}_\e}/|y_\e|^2)}{(\tau+e^{\tilde{u}_\e}/|y_\e|^2)^3}=-4\pi m_{j,2}\delta_0\ \ \textrm{on}\ \ B_{\frac{r}{\e |y_\e|}}(0).
\end{equation*}
Moreover, \begin{equation}\label{upb}
\begin{aligned}\tilde{u}_\e(x)&=\hat{u}_\e(|y_\e| x)+2\ln(|y_\e||x|)-2\ln|x|\\&\le\tilde{u}_\e\Big(\frac{y_\e}{|y_\e|}\Big)-2\ln|x|\ \ \textrm{on} \ \ B_{\frac{r}{\e|y_\e|}}(0)\setminus B_{\frac{\eta}{|y_\e|}}(0).
\end{aligned}\end{equation}
To prove the claim (\ref{claimbounderycond2}), we argue by contradiction and consider the following two cases.

Case 1: Suppose that
$$\lim_{\e\to0}\Big(e^{\hat{u}_\e(y_\e)}|y_\e|^2\Big)=+\infty.$$
Then we see that, in view of (\ref{claimbounderycond00}), we have
$\lim_{\e\to0}|y_\e|=+\infty$  and $\lim_{\e\to0}(\e |y_\e|)=0$.

Moreover, we see that
 $$s_\e\equiv\exp\Big(-\frac{1}{2}\tilde{u}_\e\Big(\frac{y_\e}{|y_\e|}\Big)\Big)\to0\ \ \textrm{as}\ \ \e\to0.$$
In view of (\ref{upb}), we see that for any $x\in B_{\frac{1}{\delta}}(0)\setminus B_\delta(0)$,
\begin{equation}\label{upb2}
\begin{aligned}\tilde{u}_\e(x)\le \tilde{u}_\e\Big(\frac{y_\e}{|y_\e|}\Big)-2\ln\delta.
\end{aligned}
\end{equation}
By using (\ref{anotherclaim}) and $\lim_{\e\to0}|y_\e|=+\infty$, we also see that
\begin{equation}
\begin{aligned}\label{L0}\lim_{\e\to0}\Big(\frac{e^{\tilde{u}_\e(x)}}{|y_\e|^2}\Big)=\lim_{\e\to0}e^{\hat{u}_\e(|y_\e|x)}=0\ \ \textrm{on}\ \  B_{\frac{1}{\delta}}(0)\setminus B_\delta(0).\end{aligned}
\end{equation}

Let $\bar{w}_\e(x)\equiv\tilde{u}_\e\Big(s_\e x+\frac{y_\e}{|y_\e|}\Big)+2\ln s_\e$ for $|x|<\frac{\delta}{2s_\e}$.
For small $\e, \delta>0$, $\bar{w}_\e$ satisfies
\begin{equation*}
\Delta \bar{w}_\e+ \frac{e^{\bar{w}_\e}(1-\frac{e^{\bar{w}_\e}}{s_\e^2|y_\e|^2})}{(\tau+\frac{e^{\bar{w}_\e}}{s_\e^2|y_\e|^2})^3}=0
\ \ \textrm{on}\ B_{\frac{\delta}{2s_\e}}(0).
\end{equation*}
By using (\ref{upb2}), we see that
\begin{equation}\label{forharnkeq}
\begin{aligned}
\left\{
 \begin{array}{ll}
\bar{w}_\e(x)\le \tilde{u}_\e(y_\e/|y_\e|)-2\ln\delta+2\ln s_\e=-2\ln\delta\ \ \textrm{for} \ \ |x|<\frac{\delta}{2s_\e},
\\ \bar{w}_\e(0)=\tilde{u}_\e(y_\e/|y_\e|)+2\ln s_\e=0.
 \end{array}\right.
\end{aligned}
\end{equation}
In view of (\ref{L0}), we also conclude that $\lim_{\e\to0}\Big(\frac{1}{s_\e^2|y_\e|^2}\Big)=0$ and for small $\e>0$,
\begin{equation}\label{forharnkeq2}
\begin{aligned}
0\le-\Delta\bar{w}_\e\le\frac{1}{\delta^2\tau^3}\ \ \textrm{on}\ B_{\frac{\delta}{2s_\e}}(0).
\end{aligned}
\end{equation}
By using (\ref{forharnkeq}), (\ref{forharnkeq2}), and Lemma \ref{harnarkineq}, we see that  for any $p>1$ and $R>0$,
there exist constants $\sigma\in(0,1)$ and $\gamma>0$,
depending on $R>0$ only such that
$$0=\bar{w}_\e(0)\le \sup_{B_R(0)} \bar{w}_\e\le\sigma\inf_{B_R(0)}\bar{w}_\e+(1+\sigma)\gamma\|\Delta \bar{w}_\e\|_{L^p(B_{2R}(0))}-2(1-\sigma)\ln\delta,$$
which implies that  $\bar{w}_\e$ is bounded in $C^0_{\textrm{loc}}(B_{\frac{\delta}{2s_\e}}(0))$.
Then there exists a functon $w_*$ such that
$\bar{w}_\e\to w_*$ in $C^2_{\textrm{loc}}(\RN)$. By Lemma \ref{bddofintegration}, $w_*$ satisfies
\begin{equation*}
\begin{aligned}
\left\{
 \begin{array}{ll}
\Delta w_*+ \frac{e^{w_*}}{\tau^3}=0
\ \textrm{in}\ \RN,
\\w_*(0)=0, \ \ e^{w_*}\in L^1(\RN).
 \end{array}\right.
\end{aligned}
\end{equation*}
However we see that $w_*$ cannot be a stable solution, which yields the desired contradiction and rules out Case 1.

Case 2: Suppose that there exists $c>0$ such that
\begin{equation}\label{ha0}e^{-c}<\lim_{\e\to0}e^{\hat{u}_\e(y_\e)}|y_\e|^2<e^c.\end{equation}
Then, in view of (\ref{claimbounderycond00}), we see that $\lim_{\e\to0}|y_\e|=+\infty$
and $\lim_{\e\to0}(\e |y_\e|)=0$.
By using (\ref{upb}) and (\ref{ha0}), we also conclude that
\begin{equation}\label{ha1}
\begin{aligned}
\left\{
 \begin{array}{ll}
\tilde{u}_\e(x)\le\tilde{u}_\e\Big(\frac{y_\e}{|y_\e|}\Big)-2\ln|x|\le c-2\ln|x|\ \ \textrm{for} \ \ x\in B_{\frac{1}{\delta}}(0)\setminus B_\delta(0),
\\ -c\le \tilde{u}_\e(y_\e/|y_\e|).
 \end{array}\right.
\end{aligned}
\end{equation}
By using (\ref{anotherclaim}) and $\lim_{\e\to0}|y_\e|=+\infty$, we also have
\begin{equation}
\begin{aligned}\label{L0'}\lim_{\e\to0}\Big(\frac{e^{\tilde{u}_\e(x)}}{|y_\e|^2}\Big)=\lim_{\e\to0}e^{\hat{u}_\e(|y_\e|x)}=0\ \ \textrm{on}\ \  B_{\frac{1}{\delta}}(0)\setminus B_\delta(0).\end{aligned}
\end{equation}
Then (\ref{L0'}) implies that  for small $\e>0$,
\begin{equation}\label{ha2}
0\le-\Delta\tilde{u}_\e=\frac{e^{\tilde{u}_\e}(1-e^{\tilde{u}_\e}/|y_\e|^2)}{(\tau+e^{\tilde{u}_\e}/|y_\e|^2)^3}\le\frac{e^c}{\delta^2\tau^3}\ \ \textrm{on}\ B_{\frac{1}{\delta}}(0)\setminus B_\delta(0).
\end{equation}
By using (\ref{ha1}), (\ref{ha2}), and Lemma \ref{harnarkineq}, we see that  for any $p>1$ and $\delta>0$,
there exist constants $\sigma\in(0,1)$ and $\gamma>0$,
depending only on $\delta>0$ such that
\begin{equation*}\begin{aligned}
-c&\le \tilde{u}_\e(y_\e/|y_\e|)\le \sup_{B_{\frac{1}{\delta}}(0)\setminus B_\delta(0)} \tilde{u}_\e
\\&\le\sigma\inf_{B_{\frac{1}{\delta}}(0)\setminus B_\delta(0)}\tilde{u}_\e+(1+\sigma)\gamma\|\Delta \tilde{u}_\e\|_{L^p(B_{\frac{2}{\delta}}(0)\setminus B_{\frac{\delta}{2}}(0))}+(1-\sigma)\Big(c-2\ln\Big(\frac{\delta}{2}\Big)\Big),
\end{aligned}\end{equation*}
which implies that $\tilde{u}_\e$ is bounded in $C^0_{\textrm{loc}}(B_{\frac{r}{\e|y_\e|}}(0)\setminus\{0\})$.
Let $$\alpha=\lim_{\delta\to0}\lim_{\e\to0}\int_{B_\delta(0)}\frac{e^{\tilde{u}_\e}(1-e^{\tilde{u}_\e}/|y_\e|^2)}{(\tau+e^{\tilde{u}_\e}/|y_\e|^2)^3}dx
=\lim_{\delta\to0}\lim_{\e\to0}\int_{B_{\delta|y_\e|}(0)}\frac{e^{\hat{u}_\e}(1-e^{\hat{u}_\e})}{(\tau+e^{\hat{u}_\e})^3}dx.$$
In view of Lemma \ref{bddofintegration}, we see that there exists a function $w_*$ such that
$\tilde{u}_\e\to w_*$ in $C^2_{\textrm{loc}}(\RN\setminus\{0\})$ and
\begin{equation*}
\begin{aligned}
\left\{
 \begin{array}{ll}
\Delta w_*+ \frac{e^{w_*}}{\tau^3}=(-\alpha-4\pi m_{j,2})\delta_0
\ \textrm{in}\ \RN,
\\ (-\alpha-4\pi m_{j,2})>-4\pi,
\\ w_*\le c-2\ln|x|, \ \ e^{w_*}\in L^1(\RN).
 \end{array}\right.
\end{aligned}
\end{equation*}
However, $w_*$ cannot be a stable solution, which yields once more a contradiction and concludes
the proof of (\ref{claimbounderycond2}) as claimed.

\textit{Step 4.} For any $d\in(0,-(m_{j,2}+\beta_{j,2}))$, there exists $r_\e:=r_\e(d)\in(0,\frac{r}{\e})$ such that
\begin{equation}\label{new}\lim_{\e\to0}\int_{B_{r_\e}(0)}f_\tau(\hat{u}_\e)dx=4\pi d.\end{equation}
Now we claim that 
\begin{equation}
\begin{aligned}\label{poeqnewclaim}
&\lim_{\e\to0}\int_{B_{r_\e}(0)}F_{2,\tau}(\hat{u}_\e)dx
=2\pi d(d+2m_{j,2}).
\end{aligned}
\end{equation}
By  using  (\ref{anotherclaim}) and (\ref{claimbounderycond2}), we see that
 $\lim_{\e\to0}r_\e=+\infty$ and $\lim_{\e\to0}(\e r_\e)=0$.
Let us consider the function $\hat{\hat{u}}_\e(x)\equiv\hat{u}_\e(r_\e x)+2\ln r_\e$ which satisfies
$$\Delta \hat{\hat{u}}_\e+ \frac{e^{\hat{\hat{u}}_\e}(1-e^{\hat{\hat{u}}_\e}/r_\e^2)}{(\tau+e^{\hat{\hat{u}}_\e}/r_\e^2)^3}=-4\pi m_{j,2}\delta_0\ \ \textrm{on}\ \ B_{\frac{r}{\e r_\e}}(0).$$
 We claim that   for any $\delta>0$, there exists $C_\delta>0$ such that
\begin{equation}
\begin{aligned}\label{difflessfi}
|\hat{\hat{u}}_\e(x_1)-\hat{\hat{u}}_\e(x_2)|\le C_\delta\ \ \textrm{for any}\ \ x_1, x_2\in B_{\frac{1}{\delta}}(0)\setminus B_\delta(0).
\end{aligned}\end{equation}
By using the Green's representation formula for a solution $u_\e$ of (\ref{maineq}) and Lemma \ref{bddofintegration}, we see that $\textrm{for any}\ \ x_1, x_2\in B_{\frac{1}{\delta}}(0)\setminus B_\delta(0)$,
\begin{equation*}
\begin{aligned}\hat{\hat{u}}_\e(x_1)-\hat{\hat{u}}_\e(x_2)&=u_\e(\e r_\e x_1)-u_\e(\e r_\e x_2)
\\&=\int_\Omega (G(\e r_\e x_1,y)-G(\e r_\e x_2,y))(-\Delta u_\e(y))dy
\\&=O(1)+\int_\Omega (G(\e r_\e x_1,y)-G(\e r_\e x_2,y))\frac{f_\tau(u_\e(y))}{\e^2}dy
\\&=O(1)+\frac{1}{2\pi}\int_\Omega \ln\Big(\frac{|\e r_\e x_2-y|}{|\e r_\e x_1-y|}\Big)\frac{f_\tau(u_\e(y))}{\e^2}dy.
\end{aligned}\end{equation*}
By the mean value theorem, there exists $\theta=\theta(\e,y)\in(0,1)$ such that
 \begin{equation}
\begin{aligned}\label{forenqq}
|\ln|\e r_\e x_2-y|-\ln|\e r_\e x_1-y||&=\frac{||\e r_\e x_2-y|-|\e r_\e x_1-y||}{\theta|\e r_\e x_2-y|+(1-\theta)|\e r_\e x_1-y|}
\\&\le\frac{|\e r_\e (x_1-x_2)|}{\theta|\e r_\e x_2-y|+(1-\theta)|\e r_\e x_1-y|}.
\end{aligned}
\end{equation}
For any $y\in\Omega\setminus B_{\frac{2\e r_\e}{\delta}}(0)$,  we have $|\e r_\e x_i-y|\ge\frac{\e r_\e}{\delta}$ for $i=1,2$.
Then by using Lemma \ref{bddofintegration} and (\ref{forenqq}), we see that
 \begin{equation}
\begin{aligned}\label{fofclaim1}\int_{\Omega\setminus B_{\frac{2\e r_\e}{\delta}}(0)}\Big| \ln\Big(\frac{|\e r_\e x_2-y|}{|\e r_\e x_1-y|}\Big)\frac{f_\tau(u_\e)}{\e^2}\Big| dy=O(1).\end{aligned}
\end{equation}
By using the fact that $|\e r_\e x_i-y|\ge\frac{\e r_\e\delta}{2}$ for $i=1,2$ for any
$y\in B_{\frac{\e r_\e\delta}{2}}(0)$ together with Lemma \ref{bddofintegration} and (\ref{forenqq}), we also have
 \begin{equation}
\begin{aligned}\label{fofclaim2}\int_{B_{\frac{\e r_\e\delta}{2}}(0)}\Big| \ln\Big(\frac{|\e r_\e x_2-y|}{|\e r_\e x_1-y|}\Big)\frac{f_\tau(u_\e)}{\e^2}\Big|dy=O(1).\end{aligned}
\end{equation}
Moreover, by using (\ref{anotherclaim}) and (\ref{claimbounderycond2}), we see that
\begin{equation}
\begin{aligned}\label{fofclaim3}
&\int_{B_{\frac{2\e r_\e}{\delta}}(0)\setminus B_{\frac{\e r_\e\delta}{2}}(0)} \ln\Big(\frac{|\e r_\e x_2-y|}{|\e r_\e x_1-y|}\Big)\frac{f_\tau(u_\e)}{\e^2}dy
\\&=\int_{B_{\frac{2 r_\e}{\delta}}(0)\setminus B_{\frac{ r_\e\delta}{2}}(0)} \ln\Big(\frac{|r_\e x_2-y|}{|r_\e x_1-y|}\Big)f_\tau(\hat{u}_\e(y))dy
\\&\le2r_\e|x_1-x_2|\sup_{B_{\frac{2 r_\e}{\delta}}(0)\setminus B_{\frac{ r_\e\delta}{2}}(0)}(|f_\tau(\hat{u}_\e)|)
\int_{B_{\frac{4 r_\e}{\delta}}(0)} \frac{1}{|y|}dy
\\&=\frac{16\pi|x_1-x_2|}{\delta}\sup_{B_{\frac{2 r_\e}{\delta}}(0)\setminus B_{\frac{ r_\e\delta}{2}}(0)}(r_\e^2|f_\tau(\hat{u}_\e)|)
\\&\le\frac{32\pi|x_1-x_2|}{\delta\tau^3}\sup_{B_{\frac{2 r_\e}{\delta}}(0)\setminus B_{\frac{ r_\e\delta}{2}}(0)}(r_\e^2e^{\hat{u}_\e})=O(1).
\end{aligned}\end{equation}
At this point, (\ref{difflessfi}) follows by using
(\ref{fofclaim1}), (\ref{fofclaim2}), and (\ref{fofclaim3}).

Now we fix $y_0\in\RN\setminus\{0\}$. Then, in view of (\ref{claimbounderycond2}), (\ref{new}) and (\ref{difflessfi}),
we see that there exists a function $h$ such that
$h_\e\equiv\hat{\hat{u}}_\e-\hat{\hat{u}}_\e(y_0)\to h$ in $C^2_\textrm{loc}(\RN\setminus\{0\})$ and
$$\Delta h=-4\pi(m_{j,2}+d)\delta_0\ \ \ \textrm{in}\ \ \RN.$$
We also conclude that the function $v(x)=h(x)+2(m_{j,2}+d)\ln|x|$ satisfies
\begin{equation}\label{forfutu}
\Delta v=0\ \ \textrm{in}\ \ \RN.
\end{equation}

We consider the function $v_\e(x)\equiv h_\e(x)+2m_{j,2}\ln|x|$ which satisfies
$$\Delta v_\e(x)+ \frac{r_\e^2e^{\hat{u}_\e(r_\e x)}(1-e^{\hat{u}_\e(r_\e x)})}{(\tau+e^{\hat{u}_\e(r_\e x)})^3}=0\ \ \textrm{on}\ \ B_{\frac{r}{\e r_\e}}(0).$$
Let $v_\e(x)=\hat{v}_\e(r_\e x)$. Then, we see that
\begin{equation}\label{preeq}
\Delta \hat{v}_\e+ \frac{e^{\hat{u}_\e}(1-e^{\hat{u}_\e})}{(\tau+e^{\hat{u}_\e})^3}=0\ \ \textrm{on}\ \ B_{\frac{r}{\e}}(0).
\end{equation}
Multiplying (\ref{preeq}) by $\nabla \hat{u}_\e\cdot x$ and integrating over $B_{r_\e}(0)$, we conclude that
\begin{equation}
\begin{aligned}\label{prepo}
&\int_{B_{r_\e}(0)}2F_{2,\tau}(\hat{u}_\e)dx
\\&=\int_{\partial B_{r_\e}(0)}\Big[\Big(\nabla\hat{v}_\e\cdot\frac{x}{|x|}\Big)(\nabla\hat{v}_\e\cdot x)-\frac{|\nabla \hat{v}_\e|^2|x|}{2}+F_{2,\tau}(\hat{u}_\e)|x|-2m_{j,2}\frac{\nabla \hat{v}_\e\cdot x}{|x|}\Big]d\sigma.
\end{aligned}
\end{equation}
Hence (\ref{claimbounderycond2}) and (\ref{forfutu}) together imply
\begin{equation}
\begin{aligned}\label{poeqnew}
&\lim_{\e\to0}\int_{B_{r_\e}(0)}2F_{2,\tau}(\hat{u}_\e)dx
\\&=\lim_{\e\to0}\int_{\partial B_1(0)}\Big[\Big(\nabla v_\e\cdot\frac{x}{|x|}\Big)(\nabla v_\e\cdot x)-\frac{|\nabla v_\e|^2|x|}{2}+r_\e^2F_{2,\tau}(\hat{u}_\e(r_\e x))-2m_{j,2}\frac{\nabla v_\e\cdot x}{|x|}\Big]d\sigma
\\&=\int_{\partial B_1(0)}\Big\{\frac{(\nabla v\cdot x-2d)^2}{|x|}-\Big(\nabla v-\frac{2dx}{|x|^2}\Big)^2\frac{|x|}{2}
-\frac{2m_{j,2}(\nabla v\cdot x-2d)}{|x|}\Big\}d\sigma\\&=4\pi d(d+2m_{j,2}),
\end{aligned}
\end{equation}
and we complete the proof of our claim (\ref{poeqnewclaim}). 
At this point, in view of (\ref{anotherclaim}), (\ref{new}) and (\ref{poeqnew}), we see that
 \begin{equation*}\begin{aligned}
4\pi d
 &=\lim_{\eta\to0}\lim_{\e\to0}\int_{B_{r_\e}(0)\setminus B_\eta(0)}f_\tau(\hat{u}_\e)dx
=\lim_{\eta\to0}\lim_{\e\to0}\int_{B_{r_\e}(0)\setminus B_\eta(0)}\frac{e^{\hat{u}_\e}}{\tau^3}dx
  \\&=\lim_{\eta\to0}\lim_{\e\to0}\int_{B_{r_\e}(0)\setminus B_\eta(0)}F_{2,\tau}(\hat{u}_\e)dx
  =2\pi d(d+2m_{j,2}),
 \end{aligned}\end{equation*}
 which implies $d+2m_{j,2}=2$.
 Since $d>0$ can be chosen arbitrarily, we obtain a contradiction which concludes
the proof of $m_{j,2}+\beta_{j,2}=0$ under the assumption (\ref{as1}).
 \end{proof}
\begin{remark}
It turns out that Lemma \ref{minusinftyintheregion} and Lemma \ref{mplusbetaeq0} yield
the following result. Suppose that $u_\e-2\ln\e$ is uniformly bounded in any compact subset
of $\Omega\setminus Z_2$ and $u_\e-2\ln\e$ converges to $w$ in $C^2_\textrm{loc}(\Omega\setminus Z_2)$ as $\e\to0$, then $0\le m_{j,2}<1$ for $1\le j\le d_2$, $N_1>N_2$, and
$w$ satisfies
$$\Delta w+\frac{e^w}{\tau^3}= 4\pi\sum^{d_{1}}_{j=1}m_{j,1}\delta_{p_{j,1}}-4\pi\sum^{d_{2}}_{j=1}m_{j,2}\delta_{p_{j,2}}
\ \textrm{on}\ \Omega.$$
During the preparation of our paper, we was informed by professors Choe and Han that they have proved a similar result,
see \cite{CH}.
\end{remark}

At this point, we are ready to prove one part of  Theorem \ref{stable}.

{\bf  Proof for Theorem \ref{stable}: stable solution $\Rightarrow$ topological solution}

To prove our theorem,  we consider the following  cases.

\textit{Case 1.} If $\tau=1$, then  Lemma \ref{minusinftyintheregion} and Lemma \ref{mplusbetaeq0}
together imply that for any $r, \eta>0$,
\begin{equation}\label{mplbetaeazero}
\begin{aligned}
\left\{
 \begin{array}{ll}
\lim_{\e\to0}\Big(\sup_{\eta\le|x-p_{j,2}|\le\frac{r}{\e}}\hat{u}_\e(x)\Big)=\lim_{\e\to0}\Big(\sup_{\eta\le|x-p_{j,2}|\le\frac{r}{\e}}u_\e(\e x)\Big)=-\infty,
\\ m_{j,2}+\beta_{j,2}=0\ \ \textrm{for all} \ \ 1\le j\le d_2.
 \end{array}\right.
\end{aligned}
\end{equation}
By using (\ref{mplbetaeazero}) and Lemma \ref{intfandcapf}, we see that
 \begin{equation}\begin{aligned}\label{as3}
&\lim_{r,\eta\to0}\Big(\lim_{\e\to0}\int_{B_{\frac{r}{\e}}(p_{j,2})\setminus B_{\eta}(p_{j,2})}f'_\tau(\hat{u}_\e)dx\Big)
=\lim_{r,\eta\to0}\Big(\lim_{\e\to0}\int_{B_{\frac{r}{\e}}(p_{j,2})\setminus B_{\eta}(p_{j,2})}\frac{e^{\hat{u}_\e}}{\tau^3}dx\Big)
\\&=\lim_{r,\eta\to0}\Big(\lim_{\e\to0}\int_{B_{\frac{r}{\e}}(p_{j,2})\setminus B_{\eta}(p_{j,2})}f_\tau(\hat{u}_\e)dx\Big)=-4\pi(m_{j,2}+\beta_{j,2})=0.
 \end{aligned}\end{equation}

If  $\lim_{\e\to0}\Big(\sup_{\Omega\setminus\cup_{j}(B_r(p_{j,2}))}w_\e\Big)=-\infty\ \ \textrm{for any small}\ \ r>0$,
 then in view of (\ref{inftyftn}) and (\ref{mplbetaeazero}), we see that
\begin{equation*}\Delta g=4\pi\sum^{d_{1}}_{j=1}m_{j,1}\delta_{p_{j,1}}-4\pi\sum^{d_{2}}_{j=1}m_{j,2}\delta_{p_{j,2}}
\ \textrm{on}\ \Omega,
\end{equation*}
thus $N_1=N_2$ which  contradicts (H1).

On the other hand, if  $w_\e\to w\ \ \textrm{in}\ \ C^2_{\textrm{loc}}(\Omega\setminus Z_2)$, then
in view of (\ref{nextforeigenvalue}) and (\ref{as3}), we see that
\begin{equation*}
\begin{aligned}
\lim_{\e\to0}\mu_\e&\le\frac{1}{|\Omega|^2}\lim_{\e\to0}\int_\Omega -\frac{f'_\tau(u_\e)}{\e^2}dx
\\&=\frac{1}{|\Omega|^2}\lim_{r\to0}\lim_{\e\to0}\Big[\int_{\Omega\setminus\cup_{j}(B_r(p_{j,2}))}\frac{e^{w_\e}(-\tau+2(\tau+1)\e^2e^{w_\e}-\e^4e^{2w_\e})}{(\tau+\e^2e^{w_\e})^4}dx
\\&-\sum_{j=1}^{d_2}\int_{B_{\frac{r}{\e}}(p_{j,2})}f'_\tau(\hat{u}_\e)dx\Big]
=-\frac{1}{|\Omega|^2}\int_{\Omega}\frac{e^{w}}{\tau^3}dx<0,
\end{aligned}
\end{equation*}
which implies  $\mu_\e<0$ for small $\e>0$. Then $u_\e$ cannot be a stable solution of (\ref{maineq})
which is once more a contradiction.

\textit{Case 2.} If $N_2>N_1$ then, in view of (H2), we have $m_{j,2}\in[0,1]$ for all $1\le j\le d_2$.
Then by using Lemma \ref{minusinftyintheregion} and Lemma \ref{mplusbetaeq0} we obtain
(\ref{mplbetaeazero}). By using the same arguments as in \textit{Case 1}, we can prove that
$u_\e$ cannot be stable solution of (\ref{maineq}). We skip the details of this part to avoid repetitions.

\textit{Case 3.} If $N_2<N_1$, then we define the following set
$$J_0\equiv\{l\ |\ 1\le l\le d_2,  \ \  \lim_{\e\to0}\Big(\sup_{|x-p_{l,2}|=\eta}\hat{u}_\e(x)\Big)=-\infty\ \}.$$
If $J_0=\{1,...,d_2\}$, then the desired conclusion will follow by the same
argument adopted in \textit{Case 1}.\\
Therefore we suppose that $J_0\neq\{1,...,d_2\}$ and define  $J_1\equiv\{1,...,d_2\}\setminus J_0\neq \emptyset$.
By using Lemma \ref{minusinftyintheregion}, we see  that $\lim_{\e\to0}\Big(\sup_{\Omega\setminus\cup_{j}(B_r(p_{j,2}))}w_\e\Big)=-\infty\ \ \textrm{for any small}\ \ r>0$.
Then by (\ref{inftyftn}), we have
$$N_2=\sum_{j\in J_0}m_{j,2}+\sum_{j\in J_1}m_{j,2}<N_1=\sum_{j\in J_0}(-\beta_{j,2})+\sum_{j\in J_1}(-\beta_{j,2}).$$
By using Lemma \ref{mplusbetaeq0}, we see that there exists $j_0\in J_1$ such that
$$m_{j_0,2}<-\beta_{j_0,2}.$$
For the sake of simplicity, we assume that $p_{j_0,2}=0$.
In view  of Lemma \ref{intfandcapf}, we see that
\begin{equation*}
\lim_{r\to0}\lim_{\e\to0}\int_{B_{\frac{r}{\e}}(0)}f_\tau(\hat{u}_\e)dx
=-4\pi(m_{j_0,2}+\beta_{j_0,2})>0.
\end{equation*}
Since $j_0\in J_1$, the same argument adopted in the proof of Lemma \ref{minusinftyintheregion}
shows that there exists a function $\hat{u}$ such that $\hat{u}_\e\to\hat{u}$ in
$C^2_{\textrm{loc}}(\RN\setminus\{0\})$ and
\begin{equation*}
\begin{aligned}
\left\{
\begin{array}{ll}
\Delta\hat{u}+f_\tau(\hat{u})=-4\pi m_{j_0,2}\delta_0\ \ \ \textrm{on}  \ \ \RN,
 \\
 \\ \lim_{|x|\to\infty}\hat{u}(x)=-\infty,
 \\
 \\ f_\tau(\hat{u})\in L^1(\RN),\ \ e^{\hat{u}}\in L^1(\RN\setminus B_1(0)).
 \end{array}\right.
\end{aligned}
\end{equation*}
 Let \begin{equation}\label{hatbeta}\hat{\beta}=\frac{1}{2\pi}\intr f_\tau(\hat{u})dx.\end{equation}
Then we conclude that
\begin{equation}\label{limlnminustwo}
\lim_{|x|\to\infty}\frac{\hat{u}(x)}{\ln|x|}=-2m_{j_0,2}-\hat{\beta}<-2.
\end{equation}
Moreover,  by using (\ref{oriassm}) and (\ref{limlnminustwo}), then the similar argument adopted in Step 1
in the proof of Lemma \ref{mplusbetaeq0} shows that there exist $\nu$ and $R_0>0$ such that
\begin{equation}
\begin{aligned}\label{strictlylessinftylater}
  \hat{u}_\e<-\nu \ \ \textrm{on}\ \  B_{\frac{r}{\e}}(0)\setminus B_{R_0}(0).
\end{aligned}
\end{equation}
Let $\bar{\hat{u}}_\e(\rho)\equiv\frac{1}{2\pi \rho}\int_{\partial B_\rho(0)}\hat{u}_\e d\sigma$. Then $\bar{\hat{u}}_\e$ satisfies
\begin{equation}\label{eq}\rho\frac{d\bar{\hat{u}}_\e}{d\rho}+\frac{1}{2\pi}\int_{B_\rho(0)}f_\tau(\hat{u}_\e)dx=-2m_{j_0,2}.
\end{equation}
Then  (\ref{hatbeta}), (\ref{limlnminustwo}), (\ref{strictlylessinftylater}) and (\ref{eq}) together imply that there exists $\sigma>0$ such that for large $\rho>0$,
\begin{equation}\label{ineq}\rho\frac{d\bar{\hat{u}}_\e}{d\rho}\le -(2+\sigma).\end{equation}
We claim that there exists a constant $C>0$ such that
\begin{equation}\label{difffinite}
|\hat{u}_\e(x)-\bar{\hat{u}}_\e(|x|)|\le C\ \  \textrm{for}\ \ x\in B_{\frac{r}{\e}}(0)\setminus B_{R_0}(0).
\end{equation}
Indeed, since $j_0\in J_1$ and in view of (\ref{prove}), we see that  $\{\hat{u}_\e\}$ is uniformly
bounded in $L^\infty_{\textrm{loc}}(B_{\frac{r}{\e}}(0)\setminus\{0\})$.  Then we have
\begin{equation*} \lim_{\e\to0}\Big(\sup_{x\in \partial B_{\frac{r}{\e}}(0)\cup\partial B_{R_0}(0)}e^{\hat{u}_\e(x)}|x|^2\Big)<+\infty.
\end{equation*}
Then, by using (\ref{strictlylessinftylater}) and the similar argument adopted in Step 3
in the proof of Lemma \ref{mplusbetaeq0}, we conclude that
\begin{equation}\label{claimbounderycond2n} \lim_{\e\to0}\Big(\sup_{x\in B_{\frac{r}{\e}}(0)\setminus B_{R_0}(0)}e^{\hat{u}_\e(x)}|x|^2\Big)<+\infty.
\end{equation}
Moreover, by using the Green's representation formula for a solution $u_\e$ of (\ref{maineq}) and by arguing as in the proof of (\ref{difflessfi}),  we obtain (\ref{difffinite}).
In view of (\ref{ineq}) and (\ref{difffinite}) we can find a constant $c>0$ such that
$$\lim_{\e\to0}\int_{B_{\frac{r}{\e}}(0)\setminus B_R(0)}f_\tau(\hat{u}_\e)dx\le cR^{-\sigma}.$$
Now we see that
\begin{equation*}
\begin{aligned}
2\pi\hat{\beta}&=\lim_{R\to\infty}\int_{|x|\le R}f_\tau(\hat{u})dx
=\lim_{R\to\infty}\lim_{\e\to0}\int_{|x|\le R}f_\tau(\hat{u}_\e)dx
\\&=\lim_{R\to\infty}\lim_{\e\to0}\Big(\int_{|x|\le \frac{r}{\e}}f_\tau(\hat{u}_\e)dx-\int_{B_{\frac{r}{\e}}(0)\setminus B_R(0)}f_\tau(\hat{u}_\e)dx\Big)
\\&=-4\pi(m_{j_0,2}+\beta_{j_0,2})>0.
\end{aligned}
\end{equation*}
Moreover, the method of moving planes to be used together with (\ref{limlnminustwo}) shows
that $\hat{u}$ is radially symmetric (see  \cite{ChL, GNN}).
Now by using  Theorem 3.4 in \cite{CHLL} and $\hat{\beta}>0$, we see  that  $\hat{u}$ cannot be stable solution.

 At this point, we complete the proof of one part of Theorem \ref{stable}: stable solution $\Rightarrow$ topological solution under the assumptions (H1-2).  \hfill$\square$

\section{Proof of Theorem \ref{stable}: topological solution $\Rightarrow$ strictly stable solution}
In this section, we prove the other implication in the statement of Theorem \ref{stable}, that is,
topological solution $\Rightarrow$ strictly stable solution.
We assume that  $u_\e$ is a sequence of topological solutions of (\ref{maineq}) with a sequence $\e>0$.
Although we use arguments similar to those in \cite{T1},
we still need to carry out a subtle analysis to control the solution's sign changes.

\begin{lemma}\label{speedofconvergence}
Let $u_\e$ be a sequence of topological solutions of (\ref{maineq}) with $\e>0$.  Then, as $\e\to0$, we have

$(i)$ $u_\e\to0$ in $C^m_{\textrm{loc}}(\Omega\setminus Z)$ for any $m\in \mathbb{Z}^+$ and faster than any  power of $\e$;

$(ii)$   $\frac{(1-e^{u_\e})^2}{\e^2(\tau+e^{u_\e})^2}\to4(\tau+1)\pi \sum_{i=1,2}\sum^{d_{i}}_{j=1}m^2_{j,i}\delta_{p_{j,i}}$, weakly in the sense of measures in $\Omega$.
\end{lemma}
\begin{proof} Let $\Omega_\delta\equiv\{x\in\Omega\ |\ \textrm{dist}(x,Z)\ge\delta\ \}$.
In view of Theorem \ref{BrezisMerletypealternatives} we have $u_\e\to 0$ uniformly on
any compact subset of $\Omega\setminus Z$ as $\e\to0$.
Then  we see that for any small $\delta>0$,
\begin{equation}\begin{aligned}\label{uesq}
\Delta(|u_\e|^2)&=2|\nabla u_\e|^2+2u_\e\Delta u_\e\\&=2|\nabla u_\e|^2+\frac{2u_\e e^{u_\e}(e^{u_\e}-1)}{\e^2(\tau+e^{u_\e})^3}\ge0\ \ \textrm{on}\ \ \Omega_\delta,
\end{aligned}\end{equation}
since $t(e^t-1)\ge0$ for any $t\in\R$.
Moreover, we see that
\begin{equation}\begin{aligned}\label{nablauesq}
\Delta(|\nabla u_\e|^2)&=\sum_{i,j=1}^22\Big|\frac{\partial^2 u_\e}{\partial x_i\partial x_j}\Big|^2+\frac{2|\nabla u_\e|^2e^{u_\e}(-e^{2u_\e}+2(\tau+1)e^{u_\e}-\tau)}{\e^2(\tau+e^{u_\e})^4}
\\&\ge\frac{2|\nabla u_\e|^2e^{u_\e}(\tau+1+o(1))}{\e^2(\tau+e^{u_\e})^4}\ge0\ \ \textrm{on}\ \ \Omega_\delta\ \ \textrm{as}\ \ \e\to0.
\end{aligned}\end{equation}
We have the following inequality,
\begin{equation}\label{ele}
\frac{|t|}{1+|t|}\le|1-e^t|\ \ \textrm{for any }\ t\in\R.
\end{equation}
By using  (\ref{squareint}), (\ref{uesq}), (\ref{ele}), and the mean value theorem,
we see that there exists a constant $c>0$ such that
\begin{equation}
\begin{aligned}\label{l1}
\sup_{\Omega_{2\delta}}(|u_\e|^2)&\le\frac{1}{|\Omega_\delta|}\int_{\Omega_\delta}|u_\e|^2dx
\\&\le\frac{(1+\|u_\e\|_{L^\infty(\Omega_\delta)})^2}{|\Omega_\delta|}\int_{\Omega_\delta}\frac{|u_\e|^2}{(1+|u_\e|)^2}dx
\\&\le\frac{(1+\|u_\e\|_{L^\infty(\Omega_\delta)})^2}{|\Omega_\delta|}\Big\|\frac{(\tau+e^{u_\e})^4}{e^{u_\e}}\Big\|_{L^\infty(\Omega_\delta)}\int_{\Omega_\delta}\frac{e^{u_\e}(1-e^{u_\e})^2}{(\tau+e^{u_\e})^4}dx
\\&\le\frac{c\e^2(1+\|u_\e\|_{L^\infty(\Omega_\delta)})^2}{|\Omega_\delta|}\Big\|\frac{(\tau+e^{u_\e})^4}{e^{u_\e}}\Big\|_{L^\infty(\Omega_\delta)},
\end{aligned}
\end{equation}
for small $\e>0$.
In view of (\ref{squareint}), (\ref{nablauesq}), and the mean value theorem,  we can find
a constant $C>0$ such that
 \begin{equation}
 \begin{aligned}\label{version2gradbdd}
\sup_{\Omega_{2\delta}}(|\nabla u_\e|^2)&\le\frac{1}{|\Omega_\delta|}\int_{\Omega_\delta}|\nabla u_\e|^2dx
\\&\le\frac{1}{|\Omega_\delta|}\Big\|\frac{(\tau+e^{u_\e})^2}{e^{u_\e}}\Big\|_{L^\infty(\Omega_\delta)}\int_{\Omega}\frac{|\nabla u_\e|^2e^{u_\e}}{(\tau+e^{u_\e})^2}dx
\\&\le \frac{C}{|\Omega_\delta|}\Big\|\frac{(\tau+e^{u_\e})^2}{e^{u_\e}}\Big\|_{L^\infty(\Omega_\delta)},
 \end{aligned}
 \end{equation}
for small $\e>0$.
Let $\phi\in C^\infty(\overline{\Omega})$ be such that $\phi=0$ in $\{x\in\Omega\ |\ \textrm{dist}(x,Z)\le\delta\ \}$,
$\phi=1$ in $\Omega_{2\delta}$ and $0\le\phi\le1$.
Since $u_\e\to 0$ uniformly on any compact subset of $\Omega\setminus Z$ as $\e\to0$, we note that there exists some constant $C_\delta>0$, independent of $\e>0$, such that
\begin{equation}\label{explinear}
\Big|\frac{1-e^{u_\e}}{\tau+e^{u_\e}}\Big|\le C_\delta|u_\e|\ \ \textrm{on}\ \ \Omega_\delta.
\end{equation}
Next, by using (\ref{l1}), (\ref{version2gradbdd}) and (\ref{explinear}), we conclude that
\begin{equation}
\begin{aligned}\label{usingl1}
&\frac{1}{\e^2}\int_{\Omega_{2\delta}}\frac{e^{u_\e}(1-e^{u_\e})^2}{(\tau+e^{u_\e})^4}dx
\\&\le\frac{1}{\e^2}\int_{\Omega}\frac{e^{u_\e}(e^{u_\e}-1)}{(\tau+e^{u_\e})^3}\Big[\frac{(e^{u_\e}-1)\phi}{(\tau+e^{u_\e})}\Big] dx
\\&=\int_\Omega\Delta u_\e\Big[\frac{(e^{u_\e}-1)\phi}{(\tau+e^{u_\e})}\Big]dx
=\int_\Omega u_\e\Delta\Big[\frac{(e^{u_\e}-1)\phi}{(\tau+e^{u_\e})}\Big] dx
\\&=\int_\Omega u_\e\Big[\Delta\Big(\frac{e^{u_\e}-1}{\tau+e^{u_\e}}\Big)\phi  +\frac{2(\tau+1)e^{u_\e}\nabla u_\e\cdot\nabla\phi}{(\tau+e^{u_\e})^2}+\frac{(e^{u_\e}-1)\Delta\phi}{(\tau+e^{u_\e})}\Big] dx
\\&=\int_\Omega \Big[\frac{-(\tau+1)e^{u_\e}|\nabla u_\e|^2\phi}{(\tau+e^{u_\e})^2}+\frac{(\tau+1)e^{u_\e}u_\e\nabla u_\e\cdot\nabla\phi}{(\tau+e^{u_\e})^2}
 +\frac{(e^{u_\e}-1)u_\e\Delta\phi}{(\tau+e^{u_\e})}\Big] dx
\\&\le \int_\Omega \Big[\frac{(\tau+1)e^{u_\e}u_\e\nabla u_\e\cdot\nabla\phi}{(\tau+e^{u_\e})^2}
+\frac{(e^{u_\e}-1)u_\e\Delta\phi}{(\tau+e^{u_\e})}\Big] dx
\\&=\int_\Omega \Big[\frac{-(\tau+1)e^{u_\e}u_\e^2\Delta\phi}{2(\tau+e^{u_\e})^2}+\frac{(\tau+1)e^{u_\e}(e^{u_\e}-\tau)u_\e^2\nabla u_\e\cdot\nabla\phi}{2(\tau+e^{u_\e})^3}
+\frac{(e^{u_\e}-1)u_\e\Delta\phi}{(\tau+e^{u_\e})}\Big] dx
\\&\le c_\delta\|u_\e\|^2_{L^2(\Omega_\delta)}\le C_\delta\e^2,
\end{aligned}
\end{equation}
for some constants $c_\delta,\ C_\delta>0$.
By a suitable iteration of (\ref{l1}), (\ref{usingl1}), and the elliptic estimates, we deduce that $(i)$ holds.
In other words,
for any small $\delta>0$  and any $m, n \in \mathbb{Z}^+$,  there exists a constant $c_{\delta,m,n}>0$ such that
\begin{equation}\label{conto0}\sup_{\Omega_{2\delta}}\Big(\sum_{|\alpha|=0}^m|D^\alpha u_\e|\Big)\le c_{\delta,m,n}\e^n.
\end{equation}

Moreover, we see that $v_\e(x)=u_\e(x)+(-1)^i2m_{j,i}\ln|x-p_{j,i}|$ satisfies
\begin{equation}\label{maineqq}
\Delta v_\e+\frac{f_\tau(u_\e)}{\e^2}=0\ \ \textrm{on}\ \ B_r(p_{j,i}).
\end{equation}
For the sake of simplicity, we assume that $p_{j,i}=0$.
Multiplying (\ref{maineqq}) by $\nabla u_\e\cdot x$ and integrating over $B_r(0)$, we obtain the Pohozaev
type identity
\begin{equation*}
\begin{aligned}
&\int_{\partial B_r(0)}\Big[\Big(\nabla v_\e\cdot\frac{x}{|x|}\Big)(\nabla v_\e\cdot x)-\frac{|\nabla v_\e|^2}{2}|x|+\frac{1}{\e^2}F_{1,\tau}(u_\e)|x|\Big]d\sigma
\\&=\int_{B_r(0)}\frac{2 F_{1,\tau}(u_\e)}{\e^2}+\frac{(-1)^{i-1}2m_{j,i}f_\tau(u_\e)}{\e^2}dx,
\end{aligned}
\end{equation*}
where $F_{1,\tau}(u)=\frac{-(1-e^u)^2}{2(\tau+1)(\tau+e^u)^2}$.
By using (\ref{conto0}), we have  $$\lim_{\e\to0}\int_{\partial B_r(0)}\frac{1}{\e^2}F_{1,\tau}(u_\e)|x|d\sigma=0,$$
thus $$\lim_{\e\to0}\int_{B_r(0)}\frac{2 F_{1,\tau}(u_\e)}{\e^2}dx=-4\pi m^2_{j,i},$$
$$\lim_{\e\to0}\int_{B_r(0)}\frac{(1-e^{u_\e})^2}{\e^2(\tau+e^{u_\e})^2}dx=4(\tau+1)\pi m^2_{j,i},$$
for any small $r>0$ which concludes the proof of our lemma.
 \end{proof}
For a solution  $u_\e$ of (\ref{maineq}), let
\begin{equation}\label{inf}
\mu_\e\equiv\inf_{\phi\in W^{1,2}(\Omega)\setminus\{0\}}
\frac{\int_{\Omega}|\nabla\phi|^2-\frac{1}{\e^2}f'_\tau(u_\e)\phi^2dx}{\|\phi\|^2_{L^2(\Omega)}},
\end{equation}
and $\phi_\e$ be the corresponding first eigenfunction with
$\phi_\e>0$ in $\Omega$ and $\|\phi_\e\|_{L^2(\Omega)}=1$,
\begin{equation}\label{mue}
\mu_\e=\int_{\Omega}|\nabla\phi_\e|^2-\frac{1}{\e^2}f'_\tau(u_\e)\phi_\e^2dx,
\end{equation}
and
\begin{equation}\label{phieeq}
-\Delta\phi_\e-\frac{1}{\e^2}f'_\tau(u_\e)\phi_\e=\mu_\e\phi_\e.
\end{equation}
We note that $\e^2\mu_\e$ is bounded from below:
$$\e^2\mu_\e\ge-\int_\Omega f'_\tau(u_\e)\phi_\e^2dx\ge-\sup_{t\in\R}|f'_\tau(t)|.$$
To prove    Theorem \ref{stable}, we argue by contradiction and suppose that, along a subsequence (still
denoted in the same way), we have a sequence of topological solutions $u_\e$ of (\ref{maineq})  with a sequence $\e>0$ such that
\begin{equation}\label{lesszero}
\lim_{\e\to0}\e^2\mu_\e=\mu_0\le0.
\end{equation}
In  view of $(i)$  of Lemma \ref{speedofconvergence} and (\ref{lesszero}), we have the following lemma.
\begin{lemma}\label{forblowup}
 There exists $p_{j_0,i_0}\in Z$ and $r_0>0$ such that for any $r\in(0,r_0)$, there exists a constant $a_r>0$ such that
$$\lim_{\e\to0}\int_{B_r(p_{j_0,i_0})}\phi_\e^2dx\ge a_r.$$
\end{lemma}
\begin{proof}
Suppose that there exists a small $r>0$ such that
\begin{equation}\label{notpeak}
\lim_{\e\to0}\int_{\cup_{j,i}B_r(p_{j,i})}\phi_\e^2dx=0.
\end{equation}
Then
\begin{equation*}
\begin{aligned}
&\lim_{\e\to0}\Big|\int_{\cup_{j,i}B_r(p_{j,i})}f'_\tau(u_\e)\phi_\e^2dx\Big|
\le\sup_{t\in\R}|f'_\tau(t)|\lim_{\e\to0}\int_{\cup_{j,i}B_r(p_{j,i})}\phi_\e^2dx=0.
\end{aligned}
\end{equation*}
By using $(i)$ of Lemma \ref{speedofconvergence}, we see that
\begin{equation}
\begin{aligned}\label{outside2}
&\int_{\Omega\setminus\cup_{j,i}B_r(p_{j,i})}f'_\tau(u_\e)\phi_\e^2dx
\\&=\int_{\Omega\setminus\cup_{j,i}B_r(p_{j,i})}\Big(-\frac{1}{(\tau+1)^3}+o(1)\Big)\phi_\e^2dx\ \ \textrm{as}\ \e\to0.
\end{aligned}
\end{equation}
Next, by using (\ref{lesszero}), (\ref{notpeak}) and (\ref{outside2}), we see that
\begin{equation*}
\begin{aligned}
0&\ge\lim_{\e\to0}\e^2\mu_\e=\lim_{\e\to0}\int_{\Omega}\e^2|\nabla\phi_\e|^2-f'_\tau(u_\e)\phi_\e^2dx
\\&\ge\lim_{\e\to0}\int_{\Omega}-f'_\tau(u_\e)\phi_\e^2dx
=\frac{1}{(\tau+1)^3}.
\end{aligned}
\end{equation*}
This is the desired contradiction which concludes the proof of our lemma.
\end{proof}
Since $f_\tau(u)=-f_{\tau^{-1}}(-u)/\tau^3$,  we can assume without loss of generality
that $i_0=2$,  $p_{j_0,i_0}=0$, and $\nu\equiv m_{j_0,i_0}$ in Lemma \ref{forblowup}.
We consider the scaled function
\begin{equation}\label{hatu}
\hat{u}_\e(y)=u_\e(\e y)\ \textrm{in}\ B_{\frac{r_0}{\e}}(0).
\end{equation}
Then $\hat{u}_\e$ satisfies
$$\Delta \hat{u}_\e+ \frac{e^{\hat{u}_\e}(1-e^{\hat{u}_\e})}{(\tau+e^{\hat{u}_\e})^3}=-4\pi\nu\delta_0
\ \textrm{in}\ B_{\frac{r_0}{\e}}(0).$$
 Now we have the following lemma.
\begin{lemma}\label{linfty}
$\lim_{\e\to0}(\sup_{B_{\frac{r_0}{\e}}(0)}|\hat{u}_\e-u|)=0$, where $u$ is  a topological solution of
\begin{equation}\begin{aligned} \label{u}
\left \{
\begin{array}{ll}
\Delta u+ \frac{e^{u}(1-e^{u})}{(\tau+e^{u})^3}=-4\pi\nu\delta_0
\ \textrm{in}\ \RN,
\\ \sup_{\RN\setminus B_1(0)}|\nabla u|<+\infty,
\\ \frac{e^{u}(1-e^{u})}{(\tau+e^{u})^3},\ \frac{(1-e^{u})^2}{(\tau+e^{u})^2}\in L^1(\RN).
\end{array}\right.
\end{aligned}\end{equation}
\end{lemma}
\begin{proof}
We decompose \begin{equation}\label{decomposehatu}\hat{u}_\e(y)=-2\nu\ln|y|+\hat{v}_\e(y).\end{equation}
Then $\hat{v}_\e$ satisfies
\begin{equation}\label{hatve}
\Delta \hat{v}_\e+ \frac{|y|^{-2\nu} e^{\hat{v}_\e}(1-|y|^{-2\nu}e^{\hat{v}_\e})}{(\tau+|y|^{-2\nu}e^{\hat{v}_\e})^3}=0
\ \textrm{in}\ B_{\frac{r_0}{\e}}(0).
\end{equation}
By using Lemma \ref{speedofconvergence}, $\lim_{x\to p_{j,2}}u_\e(x)=+\infty$  and   the  maximum principle,
we conclude that  there exists $c>0$ such that for small $\e>0$,
 \begin{equation}\label{zerolesslarger}
\inf_{B_r(p_{j,2})}u_\e\ge-c.
\end{equation}
In view of (\ref{decomposehatu}) and (\ref{zerolesslarger}), we  have
$$\hat{v}_\e\Big|_{\partial B_{R}(0)}\ge -c+2\nu\ln R\  \ \textrm{for any}\ R>0.$$
 By using the Green's representation formula for a solution $u_\e$ of (\ref{maineq})
(see (\ref{greenrp}) and  (\ref{gradgrad})), we see that there exists $c_0>0$ such that
\begin{equation}\label{gradgrad2}
|\nabla \hat{v}_\e(x)|\le c_0 \ \ \textrm{on}\ \ B_{\frac{r_0}{\e}}(0).
\end{equation}

We claim that $\hat{v_\e}$ is uniformly bounded in the $C^{2,\alpha}$ topology. To prove our claim,
we argue by contradiction and suppose that there exists $R_0>0$ such that
$\lim_{\e\to0}\Big(\sup_{B_{R_0}(0)}\hat{v}_\e\Big)=+\infty$. Then (\ref{gradgrad2})
implies that $\lim_{\e\to0}\Big(\inf_{B_{R}(0)}\hat{v}_\e\Big)=+\infty$ for any  $R\ge R_0$.
Clearly Lemma \ref{speedofconvergence} shows that, for any $R\ge R_0$,
\begin{equation}\label{large}4(\tau+1)\pi \nu^2\ge\lim_{\e\to0}\int_{B_R(0)}\frac{(1-|x|^{-2\nu}e^{\hat{v}_\e})^2}{(\tau+|x|^{-2\nu}e^{\hat{v}_\e})^2}dx=\pi R^2.\end{equation}
Since the right hand side of (\ref{large}) could be arbitrarily large, we obtain a contradiction which
 proves our claim.

Then we obtain a subsequence $\hat{v}_\e$ (still denoted in the same way) such that
\begin{equation}\label{convergencevep}
\hat{v}_\e\to v\ \textrm{uniformly in}\ C^2_{\textrm{loc}}(\RN).
\end{equation}
Let us define $u(y)\equiv-2\nu\ln|y|+v(y)$. In view of  (\ref{gradgrad2}),
Lemma \ref{bddofintegration} and  Lemma \ref{speedofconvergence}, we see that $u$ satisfies (\ref{u}).
Since $\sup_{\RN\setminus B_1(0)}|\nabla u|<+\infty$ and
$\frac{(1-e^{u})^2}{(\tau+e^{u})^2}\in L^1(\RN)$, we  see that $u$ is a topological solution in $\RN$.
Moreover, by using a Pohozaev type identity (see Lemma \ref{speedofconvergence}),
we  have
\begin{equation}\label{intr}\intr \frac{(1-e^{u})^2}{(\tau+e^{u})^2}dx=4(\tau+1)\pi \nu^2.\end{equation}

Now we claim that a stronger convergence property holds, namely $$\lim_{\e\to0}(\sup_{B_{\frac{r_0}{\e}}(0)}|\hat{u}_\e-u|)=0.$$
In view of (\ref{convergencevep}), we have
\begin{equation}\label{est1}
\lim_{\e\to0}\Big(\sup_{B_1(0)}|\hat{u}_\e-u|\Big)=0.
\end{equation}

We also see that
\begin{equation*}
\begin{aligned}
&\int_{B_{\frac{r_0}{\e}}(0)}\frac{(e^{\hat{u}_\e}-e^u)^2}{(\tau+e^{\hat{u}_\e})^2}dx
=\int_{B_{\frac{r_0}{\e}}(0)}\frac{(e^{\hat{u}_\e}-1)^2}{(\tau+e^{\hat{u}_\e})^2}
+\frac{(e^u-1)^2}{(\tau+e^{\hat{u}_\e})^2}
-\frac{2(1-e^{\hat{u}_\e})(1-e^u)}{(\tau+e^{\hat{u}_\e})^2}dx.
\end{aligned}
\end{equation*}
At this point Lemma \ref{speedofconvergence}, (\ref{intr}), and the dominated convergence theorem
together imply that
\begin{equation}\label{est3}
\lim_{\e\to0}\int_{B_{\frac{r_0}{\e}}(0)}\frac{(e^{\hat{u}_\e}-e^u)^2}{(\tau+e^{\hat{u}_\e})^2}dx
=0.
\end{equation}
By using (\ref{u}), (\ref{zerolesslarger}), (\ref{gradgrad2}), (\ref{est1}),  and (\ref{est3}),
we obtain the desired conclusion.
\end{proof}
At this point, we are ready to prove  Theorem \ref{stable}.

{\bf  Proof for Theorem \ref{stable}: topological solution $\Rightarrow$ strictly stable solution}

In view of the strong convergence property as stated in Lemma \ref{linfty},
we can deduce  information about the limiting problem of the linearized equation of (\ref{maineq}) at $u_\e$.
With this purpose,
we define
\begin{equation}
\label{defpsie}\hat{\psi}_\e(y)\equiv\e\phi_\e(\e y)\ \ \textrm{on}\ \ B_{\frac{r_0}{\e}}(0).
\end{equation} Then we have
\begin{equation}\begin{aligned} \label{psieeq}
\left \{
\begin{array}{ll}
-\Delta\hat{\psi}_\e-f'_\tau(\hat{u}_\e)\hat{\psi}_\e=\e^2\mu_\e\hat{\psi}_\e\ \textrm{on}\ B_{\frac{r_0}{\e}}(0),
\\ \hat{\psi}_\e>0\ \textrm{in}\ B_{\frac{r_0}{\e}}(0),
\end{array}\right.
\end{aligned}\end{equation}
and $\|\nabla \hat{\psi}_\e\|_{L^2(B_{\frac{r_0}{\e}}(0))}+\|\hat{\psi}_\e\|_{L^2(B_{\frac{r_0}{\e}}(0))}\le C$ for some constant $C>0$.
 By using standard elliptic estimates, we see that $\hat{\psi}_\e$ is uniformly bounded in the $C^{2,\alpha}_{\textrm{loc}}$ topology.
Hence, by passing to a subsequence (still denoted in the same way), we see that there exists $\hat{\psi}\ge0$ such that
$$\hat{\psi}_\e\to\hat{\psi}\ \ \textrm{in}\  C^2_{\textrm{loc}}(\RN),$$
and
\begin{equation}\begin{aligned} \label{psieq}
\left \{
\begin{array}{ll}
-\Delta\hat{\psi}-f'_\tau(u)\hat{\psi}=\mu_0\hat{\psi}\ \textrm{in}\ \RN,
\\ \hat{\psi}\in W^{1,2}(\RN),\  \ \hat{\psi}\ge0.
\end{array}\right.
\end{aligned}\end{equation}

Since  $u$ has exponential decay at infinity, then $f'_\tau(u)+\frac{1}{(\tau+1)^3}$ has exponentially decay at infinity. Hence by Lemma  \ref{linfty}, we see that
 \begin{equation*}\begin{aligned}
&\lim_{\e\to0}\int_{B_{\frac{r_0}{2\e}}(0)}\Big(f'_\tau(\hat{u}_\e)+\frac{1}{(\tau+1)^3}\Big)\hat{\psi}_\e^2dx
=\int_{B_R(0)}\Big(f'_\tau(u)+\frac{1}{(\tau+1)^3}\Big)\hat{\psi}^2 dx+O(e^{-\delta_0R})
\end{aligned}\end{equation*}
  for some $\delta_0>0$. Hence by using (\ref{defpsie}), (\ref{psieeq}),  and Lemma \ref{forblowup}, we can prove that for large $R>0$,
 \begin{equation*}\begin{aligned}
&\int_{B_R(0)}\Big(f'_\tau(u)+\frac{1}{(\tau+1)^3}\Big)\hat{\psi}^2dx
\\&\ge\lim_{\e\to0}\Big(-\e^2\mu_\e+\frac{1}{(\tau+1)^3}\Big)\int_{B_{\frac{r_0}{2\e}}(0)}\hat{\psi}_\e^2dx+O(e^{-\delta_0R})
\\&\ge \Big(|\mu_0|+\frac{1}{(\tau+1)^3}\Big)a_{\frac{r_0}{2}}+O(e^{-\delta_0R})>0,
\end{aligned}\end{equation*}
which implies $\hat{\psi}\neq0\in W^{1,2}(\RN)$ (see Lemma 4.15 in \cite{T1} for further details).
On the other side, by arguing as in Proposition 4.16 in \cite{T1}, we see that the problem (\ref{psieq}) admits only the trivial solution and we obtain a contradiction.
This observation concludes
the proof of Theorem \ref{stable}: topological solution $\Rightarrow$ strictly stable solution. $\square$

\section{Uniqueness of stable solution}
In this section, we deduce Theorem \ref{uniqueness} from Theorem \ref{stable}.

{\bf Proof of Theorem \ref{uniqueness}} The existence of stable solution can be proved by
well known monotone iteration schemes and therefore we will skip it here.
Hence, to prove Theorem \ref{uniqueness}, it suffices to prove the uniqueness property.
We argue by contradiction and suppose that there exist two sequences of distinct stable solutions $u_{\e,1}$
and $u_{\e,2}$ of (\ref{maineq}). From Theorem \ref{stable},  up to the extraction of subsequences,
we have $u_{\e,i}\to 0$ uniformly in any compact subset of $\Omega\setminus Z$ as $\e\to0$ for $i=1, 2$.
Since $u_{\e,1}-u_{\e,2}$ is not identically zero, we can define
$\phi_\e\equiv\frac{u_{\e,1}-u_{\e,2}}{\|u_{\e,1}-u_{\e,2}\|_{L^2(\Omega)}}$  which satisfies
$$\Delta \phi_\e+\frac{1}{\e^2}f'_\tau(\eta_\e)\phi_\e=0\ \  \textrm{on}\ \ \Omega,$$
where $\eta_\e$ is some real number between $u_{\e,1}$ and $u_{\e,2}$.
By using  the proof of Lemma \ref{forblowup}, we see that
there exist $p_{j_0,i_0}\in Z$ and $r_0>0$ such that for any $r\in(0,r_0)$, there exists a constant $a_r>0$ such that
$$\lim_{\e\to0}\int_{B_r(p_{j_0,i_0})}\phi_\e^2dx\ge a_r.$$
Since $f_\tau(u)=-f_{\tau^{-1}}(-u)/\tau^3$, we can assume without loss of generality that $i_0=2$,
$p_{j_0,i_0}=0$, and $\nu\equiv m_{j_0,i_0}$.
We consider the scaled function $$\hat{u}_{\e,i}(y)=u_{\e,i}(\e y)\ \textrm{in}\ B_{\frac{r_0}{\e}}(0)\equiv\Big\{y\in\RN\ \Big|\ |y|<\frac{r_0}{\e}\Big\}.$$
In view of Lemma \ref{linfty},
we obtain
$$\hat{u}_{\e,i}\to u_i\ \textrm{uniformly in}\ C^2_{\textrm{loc}}(\RN) \ \ \textrm{for} \ \ i=1,2,$$
 where $u_i$ is a topological solution of
$$\Delta u_i+ \frac{e^{u_i}(1-e^{u_i})}{(\tau+e^{u_i})^3}=-4\pi\nu\delta_0
\ \textrm{in}\ \RN.$$ Moreover, we can apply the method of moving planes (see \cite{GNN, H}) to conclude that $u_i$ is radially symmetric about the origin.
Since radially symmetric and topological solutions are unique (see \cite{CHLL}), we conclude that
$u_1=u_2$ in $\RN$. Let us set $u\equiv u_1$.
We can find $\hat{\psi}$ such that
$$\e\phi_\e(\e y)\to\hat{\psi}(y)\ \ \textrm{in}\  C^2_{\textrm{loc}},$$
and
\begin{equation*}\begin{aligned}
\left \{
\begin{array}{ll}
-\Delta\hat{\psi}-f'_\tau(u)\hat{\psi}=0\ \textrm{in}\ \RN,
\\ \hat{\psi}\in W^{1,2}(\RN).
\end{array}\right.
\end{aligned}\end{equation*}
By arguing as in the proof of Theorem \ref{stable} (see Section 5), we see that $\hat{\psi}\neq0$.
Then, \begin{equation}\label{infweight}
\mu^*\equiv\inf_{\psi\in W^{1,2}(\RN)\setminus\{0\}}\frac{\intr |\nabla \psi|^2-f'_\tau(u)\psi^2dx}
{\intr (1-e^u)\psi^2dx}\le0.
\end{equation}
Then Lemma \ref{linfty} shows that the infimum of (\ref{infweight}) is attained at some
$\psi_0\in W^{1,2}(\RN)\setminus\{0\}$ satisfying
\begin{equation*}
-\Delta\psi_0-f'_\tau(u)\psi_0=\mu^*(1-e^u)\psi_0,\  \psi_0>0\ \ \textrm{in}\ \RN.
\end{equation*}
At this point Theorem 3.4 in \cite{CHLL} shows that  $\mu^*<0$. However,
 by arguing as in Proposition 4.16 in \cite{T1}, we can show that $\psi_0\equiv0$ which is the desired contradiction.
Therefore there exists a unique stable solution of (\ref{maineq}) for sufficiently small  $\e>0$.
$\square$

\section{Appendix}
In this section,  we discuss nontopological solutions of the following equation:
\begin{equation}\begin{aligned}\label{limitnontopological}
\left \{
\begin{array}{ll}
\Delta u +f_\tau(u)=4\pi\nu\delta_0\ \  \textrm{in}\ \ \RN,
\\ f_\tau(u)\in L^1(\RN),
\\ \lim_{|x|\to\infty}u(x)=-\infty.
\end{array}\right.
\end{aligned}\end{equation}
As we mentioned in Section 4, we need to analyze a solution $u_\e$ of (\ref{maineq}), such that $u_\e-2\ln\e$ has a
bubble at some point in $\Omega\setminus Z_2$ and
$u_\e$ (after a suitable scaling) tends to a nontopological solution $u$ of (\ref{limitnontopological}).
It is not difficult to check that it is enough to our purposes to consider the case $\nu\ge0$.
Concerning this problem, we have the following proposition.
\begin{prop} Let $u$ be a solution of (\ref{limitnontopological}) and $\nu\ge0$. Then $u$ is unstable.
\end{prop}
\begin{proof}By using the maximum principle, we always have $u<0$.
Moreover, if $u$ is radially symmetric, then Theorem 3.4 in \cite{CHLL} shows that $u$ is unstable.
In particular, if $\nu=0$,
then Lemma \ref{radiallysymm} shows that $u$ is  a radially symmetric function.
Thus, we only need to prove the instability of $u$ in the case where $\nu>0$ and $u$ is not radially symmetric.
Let us set
$$\frac{\partial}{\partial \theta}=x_2\frac{\partial}{\partial x_1}-x_1\frac{\partial}{\partial x_2}.$$
Then we see that $$\Delta(\partial_\theta u) +f'_\tau(u)(\partial_\theta u)=0\ \  \textrm{in}\ \ \RN.$$
Let $\beta=\frac{1}{2\pi}\intr f_\tau(u)dx$. Since $u<0$, we see that $e^u\in L^1(\RN)$ and $\lim_{|x|\to\infty}\frac{u(x)}{\ln |x|}=-\beta+2\nu<-2$. Moreover, by using the results in \cite{ChL}, we obtain the sharper estimate
$u(x)=(-\beta+2\nu)\ln |x|+C+O(|x|^{-\gamma})$, $u_\theta(x)=O(|x|^{-1})$ as   $|x|\to+\infty$  where $C$ is a constant and $\gamma$ is a positive constant.
We also note that there exist a  local maximum point and  a local minimum point of $u$ on each sphere of radius $r$ since $u$ is not radially symmetric.
Thus $\partial_\theta u$ changes signs, which implies at least that the first eigenvalue of the linearized equation of (\ref{maineq}) at $u$ is negative. Therefore we see that
$u$ is unstable which was the desired conclusion.
\end{proof}

 \end{document}